%% file: main.tex
\documentclass[a4paper]{article}

\include{mathscommands} 
\include{thmdefs}       
\include{sectiondefs}   

\title{The computation of Stiefel-Whitney classes} 

\author{Pierre Guillot}

\begin{document}

\maketitle
\begin{flushright}
{\em To the memory of Charles Thomas}
\end{flushright}

\input{abstract}
\pagebreak 
\setcounter{tocdepth}{2}
\tableofcontents
\input{intro}

\input{formal}

\input{algo}
\input{results}

\input{milnor}

\appendix 
\titleformat{\section}{\bf\center}{Appendix:}{1em}{}
\titleformat{\subsection}{\bf}{\Alph{section}.\arabic{subsection}.}{.5em}{}
\input{theory}

\bibliography{myrefs}
\bibliographystyle{siam}
\end{document}

%% file: mathscommands.tex
\usepackage{amsmath, amssymb, amsfonts, amscd, amsthm}
\usepackage[all]{xy}

\newcommand{\z}{\mathbb{Z}}
\renewcommand{\c}{\mathbb{C}}
\renewcommand{\r}{\mathbb{R}}
\newcommand{\q}{\mathbb{Q}}
\newcommand{\h}{\mathbb{H}}

\newcommand{\f}{\mathbb{F}}
\newcommand{\w}{\mathcal{W}^*}
\newcommand{\C}{\mathcal{C}^*}
\newcommand{\ideal}[1]{\langle #1 \rangle}

\newcommand{\ch}{\mathfrak{Ch}^*}

\newcommand{\ob}{\mathcal{O}}
\newcommand{\tob}{\tilde\ob}

%% file: thmdefs.tex
\newtheoremstyle{pedro}{}{}{\itshape}{}{\sc}{~--}{ }{\thmname{#1}\thmnumber{ #2}\thmnote{ (#3)}}

\newtheoremstyle{pedroex}{}{}{}{}{\sc}{~--}{ }{\thmname{#1}\thmnumber{ #2}\thmnote{ (#3)}}

\theoremstyle{pedro}
\newtheorem{lem}{Lemma}[section]

\newtheorem{prop}[lem]{Proposition}

\theoremstyle{remark}

\newtheorem{rmk}[lem]{Remark}

\theoremstyle{pedroex}

\newtheorem{ex}[lem]{Example}

%% file: sectiondefs.tex
\usepackage{titlesec}

\titleformat{\section}{\bf\center}{\S \arabic{section}.}{1em}{}
\titleformat{\subsection}{\bf}{\arabic{section}.\arabic{subsection}.}{.5em}{}
\titleformat{\subsubsection}[runin]{\bf}{$\quad\diamond$}{0.3em}{}

%% file: abstract.tex
\begin{abstract}

L'anneau de cohomologie d'un groupe fini, modulo un nombre premier, peut \^etre calcul\'e \`a l'aide d'un ordinateur, comme l'a montr\'e Carlson. Ici ``calculer'' signifie trouver une pr\'esentation en termes de g\'en\'erateurs et relations, et seul l'anneau (gradu\'e) sous-jacent est en jeu. Nous proposons une m\'ethode pour d\'eterminer certains \'el\'ements de structure suppl\'e\-men-taires: classes de Stiefel-Whitney et op\'erations de Steenrod. Les calculs sont concr\`etement men\'es pour une centaine de groupes (les r\'esultats sont consultables en d\'etails sur Internet).

Nous donnons ensuite une application: \`a l'aide des nouvelles informations obtenues, nous pouvons dans de nombreux cas d\'eterminer quelles sont les classes de cohomologie qui sont support\'ees par des cycles alg\'ebri\-ques.

\end{abstract}

\begin{abstract}

The cohomology ring of a finite group, with coefficients in a finite field, can be computed by a machine, as Carlson has showed. Here ``compute'' means to find a presentation in terms of generators and relations, and involves only the underlying (graded) ring. We propose a method to determine some of the extra structure: namely, Stiefel-Whitney classes and Steenrod operations. The calculations are explicitly carried out for about one hundred groups (the results can be consulted on the Internet).

Next, we give an application: thanks to the new information gathered, we can in many cases determine which cohomology classes are supported by algebraic varieties.

\end{abstract}

%% file: intro.tex
\section{Introduction}

\subsection{Computer calculations \& Stiefel-Whitney classes}

For a long time, it was very common for papers on group cohomology to
point out the lack of concrete, computational examples in the subject
(see for example the introduction to \cite{thomas}). Since then, the
situation has dramatically changed with the observation by Carlson
(see \cite{carl-visual}) that the cohomology ring could be computed in
finite time, by an algorithmic method for which a computer could be
trusted. The reader can check on the Internet (see \cite{carl-url} and
\cite{green-url})
the myriad of examples of cohomology rings which have now been
obtained. 

The question arises then: can we exploit those calculations to tackle
some problems related to the cohomology of groups ? The particular
problem which originally motivated me (and which, as it turned out,
was to play only a secondary role in this paper) was the
following. Since Totaro's paper \cite{totaro}, it is known that the
classifying space $BG$ of a finite group $G$ is a limit of algebraic
varieties (say, over $\c$), and thus one can ask for a description of
the image of the map
$$ CH^*BG \to H^*(BG, \z)  $$
where $CH^*BG$ is the {\em Chow ring} of $BG$. It is similar to the
question posed by the Hodge conjecture, but with some distinctive
features (for example $CH^*BG$ is all torsion when $G$ is finite, so
we cannot be content with a description of the map above after
tensoring with $\q$).

However, a description of $H^*(BG, \f_p)$ as a ring, which is what the
computer provides, if of little help vis-\`a-vis this problem, and
many others. In any case, let us compare the sort of output produced
by the computer with a more traditional answer.

Let us focus on the example of $Q_8$, the quaternion group of order
$8$. At the address \cite{carl-url}, one will find that $H^*(BQ_8,\f_2)$
is an algebra on generators $z, y, x$ of degree $1, 1, 4$
respectively, subject to the relations $z^2 + y^2 + zy=0$ and
$z^3=0$. One also finds a wealth of information on subgroups of $Q_8$
and their cohomology, calculations of transfers and restrictions, as
well as a thorough treatment of the commutative algebra of
$H^*(BQ_8,\f_2)$ (nilradical, Krull dimension, etc).

On the other hand, if we look at the computation by Quillen of the
cohomology of extraspecial groups (see \cite{quillenextra}), one finds
in the case of $Q_8$:
\begin{prop}\label{prop-q8}
There are $1$-dimensional, real representations $r_1$ and $r_2$ of
$Q_8$, and a $4$-dimensional representation $\Delta$, such that
$H^*(BQ_8, \f_2)$ is generated by $w_1(r_1)$, $w_1(r_2)$ and
$w_4(\Delta)$. The ideal of relations is generated by $R = w_1(r_1)^2
+ w_1(r_2)^2 + w_1(r_1)w_1(r_2)$ and $Sq^1(R)$.

Finally, $Sq^1(\Delta) = Sq^2(\Delta) = Sq^3(\Delta) = 0$. 
\end{prop}

This calls for several comments. First, if $r$ is any real
representation of a (finite or compact Lie) group $G$, then $r$ can be
seen as a homomorphism $r: G \to O(n)$ where $n$ is the real
dimension of $r$. This yields a continuous map $Br: BG \to BO(n)$
and thus a ring homomorphism $Br^*: H^*(BO(n), \f_2) \to H^*(BG,
\f_2)$. The ring $H^*(BO(n),\f_2)$ is polynomial on variables $w_1,
\ldots, w_n$, and the element $Br^*(w_i)$ is written $w_i(r)$ and
called the $i$-th {\em Stiefel-Whitney class} of $r$, a central object
of study in this paper (more details on the definition follow).

Second, the cohomology ring of any space in an {\em unstable algebra},
and is acted on by the {\em Steenrod operations} $Sq^k$, $k\ge
0$. This gives much structure on the cohomology, as will be
examplified below. For the time being, we point out that the
presentation of the cohomology of $Q_8$ is simplified by the use of
Steenrod operations, in the sense that $R$ is the only significant
relation, the other one being obtained by applying $Sq^1$.

Note that these two things are related, for one knows how to compute
the Steenrod operations on $H^*(BO(n), \f_2)$ via {\em
  Wu's formula}, see \cite{milnor}. Since $Br^*$ commutes with the
$Sq^k$, one knows how these operations act on the Stiefel-Whitney
classes. Once we know that the cohomology of a group is generated by
such classes, as is the case for $Q_8$, we get all the information on
Steenrod operations for free.

Note also, finally, that Stiefel-Whitney classes give some
geometric or representation - theoretic meaning to the relations in the
cohomology of a group, in good cases. In the case of $Q_8$ thus, there
is a relation between the representations mentioned in proposition
\ref{prop-q8}, namely:
$$ \lambda^2(\Delta) = r_1 + r_2 + r_1\otimes r_2 + 3  $$
(here ``$+3$'' means three copies of the trivial representation, and
$\lambda^2$ means the second exterior power). There are formulae
expressing the Stiefel-Whitney classes of a direct sum, a tensor
product, or an exterior power: these will be recalled in section
\ref{sec-formal}. In the present case, they give $w_2(r_1 + r_2 +
r_1\otimes r_2 + 3) = w_1(r_1)^2 + w_1(r_2)^2 + w_1(r_1)w_1(r_2)$,
while $w_2(\lambda^2( \Delta))= 0$. The latter takes into account the
fact that $w_1(\Delta) = w_2( \Delta) = w_3 (\Delta) = 0$, which in
turn is a formal consequence of the fact that $\Delta$ carries a
structure of $\h$-module, where $\h$ is the algebra of
quaternions. Putting all this together, we get an ``explanation'' for
the relation $w_1(r_1)^2 + w_1(r_2)^2 + w_1(r_1)w_1(r_2) = 0$ based on
representation theory.

All this extra decoration on the cohomology ring is extremely
useful. For example if we return to the problem, already alluded to,
of computing which cohomology classes are supported by algebraic
varieties, then we have a lot to learn from this new information. The
{\em Chern classes}, which are analogous to Stiefel-Whitney classes
but related to complex representations rather than real ones, and
which can be computed mod $2$ from the Stiefel-Whitney classes, are
always supported by algebraic varieties; this gives a ``lower
bound''. On the other hand, classes coming from the Chow ring are
killed by certain Steenrod operations, and this gives an ``upper
bound''. See \S\ref{sec-chow} for details. 


{\em The main purpose of this paper is to describe a method for the
systematic computation of Stiefel-Whitney classes}, mostly with the
help of a computer. Let us describe our success in the matter.

\subsection{Overview of results}

This paper has a companion, in the form of a computer program. The
source and the results of the
computer runs can be consulted at \cite{pedro-results}. We encourage
the reader to have a look at this page now. The present paper can
largely be seen as an explanation of the program, although it can by
all means be read independently.

It is in the nature of our algorithm that it does not work in all
cases. On the brighter side, it is very much simpler than any
full-blown method for calculating Stiefel-Whitney classes in general
(see the Appendix for a discussion of possible approaches
to the general problem). Also, our basic method can be adjusted for
specific groups and made to work in new cases by small, taylored
improvements. Our original goal however was to constitute, if not a
``database'', at least a significant collection of examples (rather
than deal with a handful of important groups).  

We have focused on the groups of order dividing $64$. We
got a full answer for the 5 groups of order 8, for 13 of the 14 groups
of order 16, for 28 of the 51 groups of order 32, and for 61 of the 267
groups of order 64. Thus we were able to deal with more than 100
groups.

Obtaining a ``full answer'' means the following. When $H^*(BG, \f_2)$ is
generated by Stiefel-Whitney classes, the computer proves it, and
gives the same sort of information as in proposition
\ref{prop-q8}. When $H^*(BG, \f_2)$ is not generated by Stiefel-Whitney
classes, the answer looks as follows. Let us the consider the smallest
example, which is that of the group of order 16 whose Hall-Senior
number is 11. This group is the semidirect product $\z/8 \rtimes
\z/2$ whose centre is a $\z/4$. The computer output is:
\begin{prop}\label{prop-gp16-6}
The cohomology of $G$ is generated by $w_1(r_2)$, $w_1(r_3)$,
$w_4(r_8)$ and a class $x$ of degree 3 which is not in the subring generated by
Stiefel-Whitney classes. Here $r_2$ and $r_3$ are $1$-dimensional real
representations, while $r_8$ is a $4$-dimensional representation of
complex type. The relations are $$w_1(r_2)^2 + w_1(r_3)^2 = 0, \qquad
w_1(r_2)^2w_1(r_3) + w_1(r_2)w_1(r_3)^2 = 0,$$ $$w_1(r_2)x + w_1(r_3)x
= 0, \qquad x^2 = 0.$$

Moreover one has $Sq^1(w_4(r_8)) = 0$, $Sq^2(w_4(r_8)) =
w_4(r_8)w_1(r_2)w_1(r_3)$, and $Sq^3(w_4(r_8))=0$.

Finally, the element $x$ is the same as the $x$ in Carlson's
presentation for $H^*(G)$.
\end{prop}

The only piece of information missing is the action of the Steenrod
operations on $x$. However, one can recover this ``by hand'', knowing
that $x$ is the same as Carlon's $x$: indeed, on Carlson's page
\cite{carl-url} we see that $x$ is a transfer of an (explicitly given)
element in the cohomology of an elementary abelian $2$-subgroup of
$G$. Transfers commute with Steenrod operations, and we deduce easily
the value of $Sq^1x$ and $Sq^2x$.

The computer also provides some other details, for example all the
Chern classes, and all the other Stiefel-Whitney classes, are given in
this presentation for $H^*(G)$.

Turning to the application to algebraic cycles, there are $38$ groups
for which we describe the image of $CH^*BG \to H^*(G)$. For example
when $G=Q_8$ this image is generated by $w_1(r_1)^2$, $w_1(r_2)^2$ and
$w_4(\Delta)$. There are $62$ groups in total for which we provide
at least partial information on algebraic cycles, see
\S\ref{sec-chow}.  
\subsection{Strategy \& Organization of the paper}

Given a group $G$, we shall always assume that we have a presentation
of $H^*(BG, \f_2)$ as a ring available (as proposition
\ref{prop-gp16-6} suggests, we have chosen to get this information
from Carlson's webpage). We shall then define a ring $\w_F(G)$ as
follows. As a graded $\f_2$-algebra, $\w_F(G)$ is to be generated by
formal variables $w_j(r_i)$ where the $r_i$'s are the irreducible,
real representations of $G$. Then we impose all the relations between
these generators which the theory of Stiefel-Whitney classes predicts:
relations coming from the formulae for tensor products and exterior
powers, rationality conditions, and so on. (It is perhaps more
accurate to say that we impose all the relations that we can think
of.)

Then one has a map $a : \w_F(G) \to H^*(BG, \f_2)$ with good properties:
namely, it is an isomorphism in degree $1$, and turns the cohomology
of $G$ into a {\em finitely generated} module over $\w_F(G)$. The key
point is that, in fact, there are {\em very few} maps between these
two rings having such properties (in practice, there are so many
relations in $\w_F(G)$ that there are few well-defined maps out of
this ring anyway). 

The slight twist here is that, unlike what you might expect, {\em we
  do not compute the effect of the map $a$}. Rather, we write down an
exhaustive list of all the maps $\w_F(G)\to H^*(BG, \f_2)$ having the
same properties as $a$, and it turns out, most of the time, that all
these maps have the same kernel and ``essentially'' the same image
(the word ``essentially'' will be justified later). More often than
not, all the maps are surjective; let us assume in this introduction
that it is so for a given $G$, postponing the more difficult
cases. Since $a$ is among these maps (without our knowing which one it
is!), we know its kernel, and we have a presentation of $H^*(BG,
\f_2)$ as a quotient of $\w_F(G)$, that is a presentation in terms of
Stiefel-Whitney classes. The computation of Steenrod operations
becomes trivial.

As a toy example, we may come back to $G=Q_8$. In this case one has
$$ \w_F(G) = \frac{\f_2[w_1(r_1), w_1(r_2), w_4(\Delta)]} {(R, Sq^1(R))}  $$
where $R=w_1(r_1)^2 + w_1(r_2)^2 + w_1(r_1)w_2(r_2)$. It is apparent
that $\w_F(G)$ is abstractly isomorphic with $H^*(G,\f_2)$;
Quillen's theorem states much more specifically that the map $a$ is an
isomorphism. Our approach, reducing to something trivial here, is to
note that there are only two classes in degree $4$ in the cohomology
ring, namely $0$ and an element $x$ which generates a polynomial
ring. If the image under $a$ of the Stiefel-Whitney class
$w_4(\Delta)$ were $0$, then $H^*(G,\f_2)$ could not be of finite type
over $\w_F(G)$. Thus $a(w_4(\Delta))=x$. Since $a$ is an isomorphism
in degree $1$, it must be surjective; for reasons of dimensions it is
an isomorphism. In this fashion we recover Quillen's result from the
presentation of the cohomology as given by Carlson and a simple game
with $\w_F(G)$, and this (in spirit if not in details) is what our program will do.  Now, describing $\w_F(G)$ explicitly is extremely long if
one proceeds manually, but it is straightforward enough that a
computer can replace us.

We insist that we are not able to give an
expression for the Stiefel-Whitney classes in terms of the generators
originally given in the presentation of $H^*(BG,\f_2)$ that we start
with. Thus we do not ``compute'' the Stiefel-Whitney classes in the
sense that one might have expected. For this reason, we have found it
worthwile to collect in an Appendix a review of the methods that one
could use in order to actually perform these computations (in the
sense, say, of obtaining cocycle representatives for the
Stiefel-Whitney classes relative to a given projective
resolution). Our objective is twofold: on the one hand, we hope to
convince the reader that these computations are considerably difficult
indeed, and that we should be so lucky to have a ``trick'' to avoid
them; on the other hand, we also hope that the suggestions we make in
the Appendix will actually be useful to anyone wishing to
take up the challenge. We describe three ways to attack the
calculations, none of which I have seen presented in the literature as
a computational device (though they each rely on classical results).

It is perhaps useful at this point to comment on the logic underlying
this paper. After reflecting on the difficulties arising in the
computation of Stiefel-Whitney classes, as exposed in the Appendix,
one wishes to calculate a minimal number of them. Certainly if a
representation can be expressed in terms of others using direct sums,
tensor products and exterior powers, then there is no need to compute
its Stiefel-Whitney classes separately. The ring $\w_F(G)$ was
originally designed to keep track of all such redundancies in a
compact way. Subsequently, it has come as a genuine surprise that this
ring made the computations so much simpler that, in many cases, there
was nothing left to do.

\subsubsection{}
The paper is organized as follows. In section \ref{sec-formal}, we introduce
the ring $\w_F(G)$, whose definition is a bit lengthy. The map $a$
will appear naturally. In section \ref{sec-algo}, we describe in
details our algorithm to find a presentation for the cohomology of $G$
with the help of $\w_F(G)$, as outlined above. In section
\ref{sec-results}, we comment on the experimental results which we
have had. Finally in section \ref{sec-chow}, we apply the preceding
results to the study of the ``cycle map'' between the Chow ring and
the cohomology of $BG$.

\subsubsection{Acknowledgements.}
It is a pleasure to thank David Green and Jon Carlson for their early
interest in this work. I am also grateful to Alain Sartout and Pierre
Navarro for their patience and help with the servers here in
Strasbourg. My thanks extend to William Stein for advertising this
work on the SAGE website, and for being generally helpful on the SAGE
forum. 

This project would have been impossible to realize without the help of
many wonderful software packages. Crucial use was made of the GAP
algebra system, the Python programming language, and the SAGE
mathematical suite. Many tools coming from the GNU Free Software
Foundation have been essential, most of all the g++ compiler. (I could
also mention GNU/linux in general, emacs, and also Latex.) 

No use was made of any commercial software.

\subsubsection{}
This paper is dedicated to the memory of Charles Thomas. Charles
lectured me on Lie groups when I was a graduate student, and was
always available for challenging discussions during my PhD. He was
very fond of characteristic classes, and explained on many occasions
how the $\lambda $-ring structure on the representation ring of a
group was a much under-used tool. I hope that the computations which
follow, for which the $\lambda$-operations play such a key role, would
have had some appeal to him.

%% file: formal.tex
\section{Formal rings of Stiefel-Whitney classes}\label{sec-formal}

From now on, we shall write $H^*(G)$ for the mod $2$ cohomology of the
finite group $G$. Occasionally we may use the notation $H^*(BG)$ in
order to emphasize a topological context.

\subsection{Formal rings}
Let $r_1, \ldots, r_m$ denote the isomorphism classes of real,
irreducible representations of $G$, and let $n_i$ be the real
dimension of $r_i$. Each $r_i$ gives rise, by choice of a basis and a
$G$-invariant inner product, to a homomorphism $G\to
O(n_i)=O_{n_i}(\r)$. The latter is well-defined up to conjugacy in
$O(n_i)$, and we also use the notation $r_i$ for any choice of
homomorphism. Note that the homotopy class of $Br_i: BG \to BO(n_i)$ is
also well-defined.

Consider now the ring
$$ A_G^* = \bigotimes_{i=1}^m H^*(BO(n_i)).  $$
This $A_G^*$ is a polynomial ring on generators which we write $\bar
w_j(r_i)$, for $1\le i \le m$ and $1\le j \le n_i$. There is a natural
map
$$ \pi=\pi_G: A_G^* \to H^*(BG)  $$
obtained by tensoring together the induced maps $Br_i^*$. The image of
$\bar w_j(r_i)$ under $\pi$ is of course $w_j(r_i)$, the $j$-th
Stiefel-Whitney class of $r_i$.

There is a modest interpretation of $A^*_G$ (and $\pi$) in ``universal''
terms. For this, we need some notations. In the presence of a graded
ring $R^*$, we write $R^\times$ for the group of elements $(a_n)$ in
the product $\prod_n R^n$ such that $a_0=1$. We write such elements
$1+ a_1 + a_2 + \cdots $ and multiply them in the obvious way. Then,
writing $R_\r(G)$ for the real representation ring of $G$, the {\em
  total Stiefel-Whitney class} is the group homomorphism
$$ \begin{array}{crcl}
w : & R_\r(G) & \longrightarrow & H^\times(G), \\  
    &  \rho   & \mapsto & 1 + w_1(\rho) + w_2(\rho) + \cdots 
\end{array}$$
defined by sending the generator $r_i$ of the free abelian group
$R_\r(G)$ to $1 + w_1(r_i) + w_2(r_i) + \cdots$. This extends the
above definition of $w_j(-)$ to representations which are not
necessarily irreducible (and even to virtual representations). Of
course we could also have given an extended definition directly for an
arbitrary representation, exactly as above: it is then a nontrivial,
but very well-known, fact that the two definitions coincide.

Consider now the following diagram:

$$
\xymatrix{
    R_\r(G) \ar[r]^{f} \ar[d]_{\bar w} & R^\times  \\
    A^\times_G \ar@{.>}[ru]_{g} & 
  }
$$

Here $R^*$ is any graded ring, and $R^\times$ is as above, while $f$
is any group homomorphism such that $f(r_i)$ is zero in degrees
greater than $n_i$. The map $\bar w$ sends $r_i$ to $1 + \bar w_1(r_i
) + \bar w_2(r_i) + \cdots $. The universality of $\bar w$ can be
expressed by saying that the dotted arrow $g$ always exists, making the
triangle commute. What is more, $g$ always comes from an underlying
map of graded rings $A^*_G \to R^*$, and the latter is unique.

Taking $R^*=H^*(G)$ and $f=w$, the map $\pi$ can then be seen as being
induced by universality.

This brings us to the following definition. Any ring which is obtained
as a quotient of $A_G^*$ by an ideal contained in $\ker\pi$ will be
called a {\em formal ring of Stiefel-Whitney classes}. As the name
suggests, we shall obtain examples of such rings by looking at formal
properties of Stiefel-Whitney classes, as we have just done with the
property ``$w_j(r_i)=0$ when $j>n_i$''. Each example $F^*$ will come
equipped with a map $R_\r(G)\to F^\times$ which is universal
among certain maps, but we shall leave to the reader this
interpretation.

An extreme example of formal ring, thus, is $A_G^*/\ker\pi$, which we
denote by $\w(G)$. It can be thought of as a subring of $H^*(G)$,
namely the subring generated by all the Stiefel-Whitney
classes. Eventually we shall end up being able to compute $\w(G)$ in
many cases, and our main tool is the use of other formal rings, which
we use as approximations to $\w(G)$.

\subsection{Formal properties}\label{subsec-formal}
The definition of $A_G^*$ (in universal terms) uses only the fact that
$w_j(r_i)$ vanishes when $j$ is large, and implicitly the formula for
the Stiefel-Whitney classes of a direct sum (in that $w$ is a group
homomorphism). We shall now review the other familiar properties of
Stiefel-Whitney classes.

\subsubsection{Rationality.}
Let $V$ be a real and irreducible representation of $G$. Schur's lemma
says that $K= End_G(V)$ is a field (not necessarily
commutative). Since $K$ must contain $\r$ in its centre, it follows
that $K$ must be one of $\r$, $\c$ or $\h$. Accordingly, $V$ is said
to be of real, complex, or quaternion type.






The consequences on Stiefel-Whitney classes are as follows. If $r_i$
is of complex type, then $r_i: G\to O(n_i)$ can be factorized as a
composition $$G\to U(d_i) \to O(n_i)$$ where $U(d_i)$ is the unitary
group, $n_i=2d_i$, and the second arrow is realification. Thus we can
also write:
$$ Br_i^* : H^*BO(n_i) \to H^*BU(d_i) \to H^*BG.  $$
Since the cohomology of $BU(d_i)$ is concentrated in even degrees, we
conclude that $w_{2j+1}(r_i) = 0$ when $r_i$ is of complex type.

Similarly, when $r_i$ is of quaternion type, we have $w_j(r_i)=0$
whenever $j$ is not divisible by $4$, for the cohomology of $BSp(d_i)$
is concentrated in degrees divisible by $4$. Here $Sp(d_i)$ is the
symplectic group and $n_i = 4d_i$.

It is very easy to check whether a given representation is of complex
or quaternion type, see \cite{serre}, \S13.2. In this way we obtain with
little effort a collection of elements of the form $\bar w_j(r_i)$ in
$A_G^*$ which all belong to $\ker\pi$.

\subsubsection{}
Before we proceed, we need to recall the {\em splitting
  principle}. This says roughly that everything happens as if any
representation were a direct sum of $1$-dimensional representations,
as far as computing the Stiefel-Whitney classes goes. More precisely,
given real representations $\alpha$ and $\beta$ of dimensions
$n_\alpha$ and $n_\beta$ respectively, one may find an injection of
$H^*(G)$ into a ring in which we have factorizations
$$ w(\alpha ) = \prod_{k=1}^{n_\alpha } (1 + a_k)  $$
and
$$ w(\beta ) = \prod_{\ell=1}^{n_\beta } (1 + b_\ell)  $$
where each $a_k$ and $b_\ell$ has degree $1$. Thus one recovers
$w_n(\alpha )$ as the $n$-th elementary symmetric function in the
``roots'' $a_k$, and likewise for $\beta $. The formulae below will be
given in terms of the roots. This traditional choice avoids
introducing lots of universal polynomials with awkward names.

\subsubsection{Tensor products.}
One has the following well-known formula:
$$ w(\alpha \otimes \beta ) = 
\prod_{\tiny \begin{array}{c}
    1\le k \le n_\alpha  \\
    1 \le \ell \le n_\beta  
\end{array} } 
(1+ a_k + b_\ell). 
$$
The reader should notice that the formula is strictly associative, in
the sense that the two universal formulae for the total
Stiefel-Whitney class of $\alpha \otimes \beta \otimes \gamma $ which
you could deduce from the result above would be precisely the same.
Likewise, it is strictly commutative. The fact that the tensor product
operation is associative and commutative up to isomorphism only
guarantees, {\em a priori}, that the formula is associative and
commutative in $H^*(G)$, for all $G$; since we can consider the
universal example of orthogonal groups and their defining
representations, however, this is enough. We shall use trivial remarks of this sort without comments in
the sequel. They are of some importance nonetheless, as we sometimes
work in the ring $A_G^*$ before applying $\pi$ to reach $H^*(G)$.

To exploit this, we look at the presentation 
$$ R_\r(G)= \z[r_1, \ldots,r_m] / \mathfrak{a}.  $$
For any $x\in \mathfrak{a}$, we wish to obtain a relation $T(x)\in
A_G^*$ which lies in $\ker\pi$. We need some care to make sure that
the computation can be done in finite time, and in particular we want
to avoid the computation of inverses of elements in the group
$A_G^\times$. We proceed thus: write $x=P - Q$ where $P$ is the sum of
the terms of $x$ which have {\em positive} coefficients. Then $Q$ also
has positive coefficients. One may obtain an element $T_P$ in
$A^\times_G$ by computing the total Stiefel-Whitney class of each term
of $P$ according to the rule above for tensor products, and then
multiply out (in $A_G^\times$) the results for the various
terms. Proceed similarly for $T_Q$, working with $Q$ instead of
$P$. Then put $T(x)=T_P - T_Q$ (which is also $T_Q - T_P$ as we are in
characteristic $2$). That $T(x)\in \ker\pi$ follows from the fact that
$x=0$ in $R_\r(G)$ and the fact that the formula for tensor products
indeed holds, in $H^*(G)$.

We note that, if one writes $x=P' - Q'$ for any $P'$ and $Q'$ having
positive coefficients, then $P'=P+S$ and $Q'=Q+S$ for some polynomial
$S$. Then $T_{P'} - T_{Q'}= T_S(T_P - T_Q)$. Since $T_S$ is a unit in
every truncated ring $A^{<N}_G$, it follows that $T(x)=u(T_{P'} -
T_{Q'})$ (here $u$ is a truncation of $T_S^{-1}$). We use this in the
proof of lemma \ref{lem-tensor-rels} below.

\begin{ex}
Let $G=\z/4$. Then $G$ has three real, irreducible representations:
the trivial one, the one-dimensional representation $\alpha$ coming
from the projection $\z/4 \to \z/2$, and the $2$-dimensional
representation $\beta$ obtained by viewing $G$ as the group of $4$-th
roots of unity in $\c$.

We have
$$ R_\r(G) = \frac{\z[\alpha,\beta ]} {(\alpha ^2 - 1,~\beta ^2 - 2
  \alpha    - 2, ~\alpha \beta  -\beta )}.  $$
Consider the relation $\alpha \beta  = \beta $. The formula for tensor products
gives in this case $w(\alpha \beta ) = 1 + w_1(\beta ) +
w_1(\alpha)w_1(\beta) + w_2(\beta)$. This being equal to the total
Stiefel-Whitney class of $\beta$, we have therefore
$w_1(\alpha)w_1(\beta) = 0$. In other words
$$ T(\alpha \beta - \beta ) =   \bar w_1(\alpha)\bar w_1(\beta) \in
\ker\pi. $$
Similarly, looking at $\beta^2 = 2 \alpha + 2$ gives the relation
$w_1(\alpha)^2 = w_1(\beta)^2$, and the element $T(\beta^2 - 2 \alpha
- 2) = \bar w_1(\alpha)^2 - \bar w_1(\beta)^2$ is in $\ker\pi$. The
relation $\alpha^2 = 1$ gives nothing.

Now, the representation $\beta$ has a complex structure, of course. It
follows that $w_1(\beta)=0$, and we will find the element $\bar
w_1(\beta)$ in $\ker\pi$.

Combining all this, we see that $\ker\pi$ contains $\bar w_1(\beta)$
and $\bar w_1(\alpha)^2$. Of course the cohomology of $\z/4$ is known,
and it turns out that $\ker\pi$ is precisely generated by these two
elements. So in this simple case all of $\ker\pi$, and indeed all the
relations in the cohomology, are explained by representation theory.
\end{ex}

All the information available can be got in finite
time:
\begin{lem}\label{lem-tensor-rels}
If $x_1, \ldots, x_n$ generate $\mathfrak{a}$, then any element of the
form $T(x)$ for $x\in\mathfrak{a}$ is in the ideal generated by the
homogeneous parts of the elements
$T(x_1), \ldots, T(x_n)$ in $A_G^*$.
\end{lem}
\begin{proof}
If $x$ and $y$ are in $\mathfrak{a}$, then let $x= P - Q$ and $y=P' -
Q'$ as above. We have $T(x+y) = u(T_{P+P'} - T_{Q+Q'})$. However
$T_{P+P'}=T_PT_{P'}$, a product in $A_G^\times$ or rather a product of
non-homogeneous elements in $A_G^*$; similarly for $Q$. Thus
$$ T(x+y) = uT_P(T_{P'} - T_{Q'}) + uT_{Q'}(T_P - T_Q) = uT_PT(y) + uT_{Q'}T(x).  $$
So $T(x+y)$ is in the ideal generated by $T(x)$ and $T(y)$.

Further, $T(x) = T(-x)$ clearly.

Finally, assume $x$ is in $\mathfrak{a}$ and $y=r_k$ for some $k$. Write $x= P - Q$. Then $T(xy)=T_{Py} - T_{Qy}$. From the above we see that there is a universal polynomial
$f$ such that $T_{Py}=f(T_P^{(1)}, T_P^{(2)}, \ldots , w_1(y), w_2(y),
\ldots )$, where
$T_P^{(i)}$ is the degree $i$ homogeneous part of $T_P$; moreover the
same $f$ has also $T_{Qy}=f(T_Q^{(1)}, T_Q^{(2)}, \ldots , w_1(y),
w_2(y), \ldots )$. It is
then clear that $T(xy)$ is in the ideal generated by the various
$T_P^{(i)} - T_Q^{(i)}$, which are the homogeneous parts of $T(x)$.

This completes the proof. 
\end{proof}

\subsubsection{Exterior powers.}
We recall the following.
$$ w(\lambda^p r_i)= \prod_{1\le i_1 < \cdots < i_p \le n_i} (1 +
a_{i_1} + \cdots + a_{i_p}).  $$
So the structure of $\lambda$-ring on $R_\r(G)$ will give us relations
between the Stiefel-Whitney classes. Now, the whole $\lambda$-ring
structure is entirely described by the value of $\lambda^p(r_i)$ for
$1 \le p \le n_i$, for there are universal polynomials expressing
$\lambda^p(x+y)$ and $\lambda^p(xy)$ in terms of the various
$\lambda^r(x)$ and $\lambda^s(y)$.

A little more precisely, for each relation $\lambda^p(r_i)=P_{i,p}$
where $P_{i,p}\in\z[r_1,\ldots,r_m]$ has degree $\le 1$ and positive
coefficients, we obtain a corresponding element $L_{i,p}\in A^*_G$
which lies in $\ker\pi$ as follows: compute the total Stiefel-Whitney
class of $\lambda^p(r_i)$ in $A^\times_G$ acording to the rule above,
then compute the total Stiefel-Whitney class of $P_{i,p}$, and call
$L_{i,p}$ the difference between the two (viewed as elements of
$A^*_G$).

Consider then a presentation
$$ R_\r(G) = \z[r_i, \lambda^pr_i | 1\le i\le m,~ 1\le p\le n_i]/\mathfrak{b}. $$
Note that we may combine the formulae for tensor products and exterior
powers, and ``translate'' any relation in $\mathfrak{b}$ into a
relation in $\ker \pi$. The details should be clear by now. Then one
has (with $x_i$ as in the previous lemma):

\begin{lem}\label{lem-ext-rels}
For an element $b\in\mathfrak{b}$, call $R$ the ``translation'' of
$b=0$ into an element of $A^*_G$, following the rules above for tensor
products and exterior powers. Then $R$ is in the ideal generated by
the the homogeneous parts of the elements $L_{i,p}$ and $T(x_i)$.
\end{lem}
\begin{proof}
If we substitute $P_{i,p}$ for $\lambda^p(r_i)$ into $b$, we get a
polynomial in $\z[r_1, \ldots, r_m]$ which evaluates to $0$ in
$R_\r(G)$, that is, an element of $\mathfrak{a}$. So $b$ can be
written as the sum of an element of $\mathfrak{a}$ and an element in
the ideal generated by the elements $\lambda^p(r_i) - P_{i,p}$. The
result now follows easily by an argument as in the previous proof.
\end{proof}

\begin{rmk}
The reader who feels uncomfortable with the details of lemma
\ref{lem-tensor-rels} and \ref{lem-ext-rels} will be reassured to know
that we do not use them in the sequel, strictly speaking. They
motivate our decision to give priority to the elements $T(x_i)$ and
$L_{i,p}$, but this could have been presented as an arbitrary decision
without breaking the logic.
\end{rmk}

\begin{ex}
We return to the example of $G=Q_8$ already considered in the
introduction. This group has three $1$-dimensional, irreducible, real
representations $r_1$, $r_2$ and $r_3$, and an irreducible,
$4$-dimensional, real representation $\Delta$ of quaternion type. We
have $r_3 = r_1r_2$ which, as above, yields $w_1(r_3) = w_1(r_1) +
w_1(r_2)$. 

However, we also have
$$ \lambda^2(\Delta) = r_1 + r_2 + r_3 + 3.$$
Computing the Stiefel-Whitney classes of $\lambda^2(\Delta)$ using the formula
for exterior powers, together with the fact that $w_1(\Delta) =
w_2(\Delta)=w_3(\Delta)=0$ since $\Delta$ has quaternion type, yields
in particular $w_2( \lambda^2(\Delta) ) = 0$. On the other hand, one
finds that $w_2(r_1 + r_2 + r_3 + 3) = w_1(r_1) w_1(r_2) + w_1(r_1)
w_1(r_3) + w_1(r_2) w_1(r_3)$, and so this element must be
zero. Combined with the expression for $ w_1(r_3)$, this yields
$w_1(r_1)^2 + w_1(r_2)^2 + w_1(r_1)w_1(r_2) = 0$.

Examining the classes in degree $3$ rather than $2$ gives, similarly,
that $$w_1(r_1)^2w_1(r_2) + w_1(r_1)w_1(r_2)^2 = 0.$$

The relations obtained in degree $1$ and $4$ are redundant.

In other words, we have found the following elements in $\ker\pi$:
$$\bar w_1(\Delta), \bar w_2(\Delta), \bar w_3(\Delta), \bar w_1(r_3) -
\bar w_1(r_1) - \bar w_1(r_2),$$ $$\bar w_1(r_1)^2 + \bar w_1(r_2)^2 +
\bar w_1(r_1) \bar w_1(r_2), \bar w_1(r_1)^2\bar w_1(r_2) + \bar
w_1(r_1) \bar w_1(r_2)^2.$$ Again in this example, it turns out that
$\ker\pi$ is generated by these elements.
\end{ex}

\subsection{Chern classes}\label{sec-chern}
Everything which we have done so far can also be done with Chern
rather than Stiefel-Whitney classes, with minor
modifications. Moreover, one can draw consequences on Stiefel-Whitney
classes by looking at Chern classes, and some of this information
cannot be got otherwise. We proceed to explain this.

Let $\rho_1, \ldots, \rho_s$ denote the complex, irreducible
representations of $G$, and let $d_i$ be the complex dimension of
$\rho_i$. There is a universal ring $A^*_{\c, G}$ which is polynomial on
generators $\bar c_j(\rho_i)$ for $1\le i \le s$ and $1\le j \le
d_i$. There is a map $\sigma : A^*_{\c, G}\to H^*(G)$, and the image of $\bar
c_j(\rho_i)$ is $c_j(\rho_i)$, the $j$-th Chern class of $\rho_i$. One
can take $A^*_{\c, G}$ to be a tensor product of cohomology rings of various
classifying spaces of unitary groups, and $g$ is induced by a
collection of group homomorphisms.

We write $\C(G)$ for $A^*_{\c, G}/\ker \sigma $, and see it as the subring of $H^*(G)$
generated by all Chern classes. A quotient of $A^*_{\c, G}$ by an ideal
contained in $\ker \sigma $ will be called a formal ring of Chern
classes. We obtain examples of formal rings by using the formal
properties of Chern classes, which are identical to those of
Stiefel-Whitney classes: one only has to bear in mind that $c_j(r_i)$
has degree $2j$ and that the ``roots'' of the splitting principle have
degree $2$. Otherwise the formulae for tensor products and exterior
powers are the same.

Now, an element in $\ker \sigma $ yields an element in $\ker\pi$ according to
the following recipe. If $\rho_i$ is the complexification of a real (and
irreducible) representation $r$, then one has $c_j(\rho_i) =
w_j(r)^2$. Note that $r$ is of real type in this case. If on the other
hand, $\rho_i$ is not such a complexification, then we let $r$ denote
its realification: it is still irreducible, and of either complex or
quaternion type. In this case one has $c_j(\rho_i) = w_{2j}(r)$ (while
the odd-degree Stiefel-Whitney classes of $r$ are zero, as already
pointed out). As a result, if we formally replace each element $\bar
c_j(\rho_i)$ by either $\bar w_j(r)^2$ or $\bar w_{2j}(r)$ following
this rule, then indeed any element in $\ker \sigma $ is turned into an
element in $\ker\pi$.

It is perhaps as well to say that we have just described a map
$$ \phi: A^*_{\c, G} \to A^*_G  $$
such that $\sigma = \pi \circ \phi$. It must carry $\ker \sigma $ into
$\ker\pi$.

\subsection{Steenrod operations}\label{sec-steenrod}
The ring $A^*_G$ is naturally an {\em unstable algebra}, so we have
operations $Sq^k$ for $k\ge 0$ on it. Of course $H^*(G)$ is also an
unstable algebra, and $\pi$ is compatible with the operations. As a
result, the ideal $\ker\pi$ is stable under the Steenrod operations.

Now given any ideal $I$ in $A^*_G$, there is a unique smallest ideal
$Sq(I)$ containing it and stable under each $Sq^k$ (namely the intersection of
all such ideals). If $I\subset \ker\pi$, then $Sq(I) \subset\ker\pi$.

It is easy to compute $Sq(I)$ concretely. If $I$ is generated by $t_1,
t_2, \ldots,t_\ell$, then either $I$ is ``Steenrod stable'' or the
ideal $I_2$ generated by all elements $Sq^kt_i$ ($1\le i \le\ell$ and
$0\le k < |t_i|$) is strictly bigger than $I$. If $I_2$ is not
Steenrod stable, we get a strictly bigger ideal $I_3$ in the same
fashion, and so on. Because $A^*_G$ is noetherian, this process has to
stop, and we obtain $Sq(I)$ in finite time.

\subsection{The ring $\w_F(G)$}
We shall now describe a particular formal ring of Stiefel-Whitney
classes, to be denoted $\w_F(G)$, which combines all the relations
which we have been discussing.

We proceed as follows:
\begin{itemize}
\item First, we let $I \subset \ker\pi$ denote the ideal generated by
  all the elements $T(x_i)$ and $L_{i,p}$ as in the lemmas
  \ref{lem-tensor-rels} and \ref{lem-ext-rels}, together with all the
  ``rationality'' relations. In other words, we consider all the
  relations in $\ker\pi$ which are discussed in section \ref{subsec-formal}.
\item Similarly, we define $J\subset\ker \sigma $ using the relations
  coming from the tensor product formula and the exterior power
  formula, only with Chern rather than Stiefel-Whitney classes.
\item Next, we consider the ideal $I'$ generated by $I$ and $\phi(J)$
  (see \S\ref{sec-chern}).
\item Finally, we take $I''=Sq(I')$ as in \S\ref{sec-steenrod}. We
  define
$$ \w_F(G)= A^*_G / I''.  $$
\end{itemize}

There is a surjective map $\w_F(G) \to \w(G)$. Composing it with the
inclusion into $H^*(G)$ induced by $\pi$, we obtain a map $a: \w_F(G)
\to H^*(G)$.

\begin{prop}
The map $a$ is an isomorphism in degree $1$, and turns $H^*(G)$ into a
finitely generated $\w_F(G)$-module.
\end{prop}

\begin{proof}
The first point follows from the isomorphism $H^1(G) = Hom(G,
\z/2\z)$. For the second point, we embed $G$ into an orthogonal group
$O(n)$, and consider the fibration $O(n)/G \to BG \to BO(n)$. Since
$O(n)$ is compact, the homogeneous space $O(n)/G$ has finitely many
cells, and it follows from the Serre spectral sequence that $H^*(G)$
is finitely generated as an $H^*BO(n)$-module. {\em A fortiori}, it is
also finitely generated as a $\w_F(G)$-module.
\end{proof}

\begin{ex}\label{ex-gp16-6-replog} Let us consider the group $G$ of order 16 which appears in
  proposition \ref{prop-gp16-6}. As indicated, this group is the only
  semidirect product $\z/8 \rtimes \z/2$ whose centre is cyclic of
  order $4$. The nonzero element in the $\z/2$ factor acts on the
  $\z/8$ factor by multiplication by $5$.

There are $10$ conjugacy classes, and so $10$ complex, irreducible
representations. Leaving out the trivial one, the character table
looks like this:

$$\begin{array}{|c|c|c|c|c|c|c|c|c|c|c|}
\hline
\textnormal{Conjugacy class} & 1& 2& 3& 4& 5& 6& 7& 8& 9& 10\\
\hline
\rho_1 &     1& -1& 1&  1&  1& -1& -1&  1&  1& -1\\
\rho_2 &     1&  1& -1&  1&  1& -1&  1& -1&  1& -1\\
\rho_3 &     1& -1& -1&  1&  1&  1& -1& -1&  1&  1\\
\rho_4 &     1&  i&  1& -1&  1&  i& -i& -1& -1& -i\\
\rho_5 &     1& -i&  1& -1&  1& -i&  i& -1& -1&  i\\
\rho_6 &     1&  i& -1& -1&  1& -i& -i&  1& -1&  i\\
\rho_7 &     1& -i& -1& -1&  1&  i&  i&  1& -1& -i\\
\rho_8 &     2&  0&  0&  2i& -2&  0&  0&  0& -2i&  0\\
\rho_9 &     2&  0&  0& -2i& -2&  0&  0&  0&  2i&  0\\
\hline
\end{array}
$$
Here we have ordered the conjugacy classes arbitrarily (in fact, we
follow the choices made by the GAP computer package). The first is the
class of the unit in $G$, and the sizes of the classes are 1, 2, 2,
1, 1, 2, 2, 2, 1, 2. 

This is enough to compute the rationality according to the recipe in
\cite{serre} (Prop. 39). We find that the first three representations are the
complexifications of $r_1$, $r_2$, $r_3$, which have thus real
type. The others give irreducible, real representations of complex
type after ``realification''. We let $r_4$, $r_6$ and $r_8$ be the
real representations underlying $\rho_4$, $\rho_6$ and $\rho_8$
respectively (these are conjugated to $\rho_5$, $\rho_7$ and $\rho_9$
respectively). The irreducible, real representations of $G$ are
exactly $r_1, r_2, r_3, r_4, r_6, r_8$ together with the trivial
representation. 

Let us explore some of the relations in $\w_F(G)$. From now on, the
elements in this ring will be written $w_j(r_i)$ rather than $\bar
w_j(r_i)$, for simplicity.

We find that $r_1 = r_2 \otimes r_3$, so that $w_1(r_1) + w_1(r_2) +
w_1(r_3) = 0$. Next, we observe that $\lambda^2(r_8) = 1 + r_1 + r_2 +
r_3 + r_6$ which, taking into account that $w_1(r_8) = w_3(r_8) =
w_1(r_6) = 0$ because $r_8$ and $r_8$ have complex type, yields
$$ w_1(r_1)w_1(r_2)w_1(r_3) + w_2(r_6)[ w_1(r_1) + w_1(r_2) +
w_1(r_3) ]=0. $$
Combining this with the previous relation, we get
$$ w_1(r_2)^2w_1(r_3) + w_1(r_2)w_1(r_3)^2 = 0.   \qquad (R) $$
We also note that $\rho_1 = \rho_4\otimes_\c \rho_4$, so that
$c_1(\rho_1) = 2 c_1(\rho_4) = 0$ (mod $2$). However $c_1(\rho_1) =
w_1(r_1)^2 = w_1(r_2)^2 + w_1(r_3)^2$. So
$$ w_1(r_2)^2 + w_1(r_3)^2 = 0.  \qquad (S)$$
As it turns out, these are all the relations that we shall keep, for
we have
$$ \w_F(G) = \frac{\f_2[w_1(r_2), w_1(r_3), w_4(r_8)]} {(R, S)} .  $$
The end of the proof of this is a lengthy exercise for the reader. It
involves showing the following relations:
$$ w_2(r_4) = w_1(r_2)^2 + w_1(r_2)w_1(r_3), \quad w_2(r_6) =
w_1(r_2)w_1(r_3), $$
$$w_2(r_8) = w_1(r_2)w_1(r_3).$$ 
This explains why we keep only the three variables above in
$\w_F(G)$. Also, one should prove that all the other relations that
one throws into $\w_F(G)$ are redundant at this point, which of course
takes a lot of time (and was done with the help of a computer).

\end{ex}

%% file: algo.tex
\section{The main algorithm}\label{sec-algo}

In this section we explain in details the procedure outlined in the
introduction. 

\subsection{Notations \& Preliminaries}

\subsubsection{Choice of variables.}

We shall assume that we have for $H^*(G)$ a presentation in terms of
variables $g_1, g_2, \ldots$ and relations.

As for $\w_F(G)$, we have a canonical choice of variables which are
all of the form $\bar w_j(r_i)$, and we have computed the relations
between these which define $\w_F(G)$ in the previous section. We shall
split the variables into three sets, and we shall use the following
rule. 

Assume that $R$ is a ring with a surjective map $P:
k[X_1,\ldots,X_n]\to R$, where $k$ is any field, and let $x_i=
P(X_i)$, for each $i$. Then we shall say that $x_i$, for lack of a
better name, is a {\em polynomial variable} with respect to this presentation if there is a generating
set for $\ker P$ which consists of polynomials not involving
$X_i$. For definiteness, let us rephrase this. Starting with any
generating set for an ideal, one may compute the {\em reduced Grobner
  basis} for this ideal using Buchberger's algorithm (see
\cite{adams}) and this is unique. It is apparent that Buchberger's
algorithm does not introduce new variables, and therefore, a variable
$x_i$ is polynomial if and only if $X_i$ does not appear in any of the
elements of the reduced Grobner basis for $\ker P$. (It also follows that the order on the power products, which is needed for Buchberger's algorithm, is irrelevant here.)

If $x_1, \ldots, x_m$ are polynomial variables for $R$ with respect to
the presentation $P$ ($m\le n$), then one has, putting $S= R/(x_1,
\ldots, x_m)$, the isomorphism $R = S[x_1,\ldots,x_m]$.

We apply this to $\w_F(G)$ and its presentation as a quotient of
$A^*_G$. We shall write $t_1, t_2, \ldots$ for the degree $1$
variables. As for the other variables, we write $p_1, p_2, \ldots$ for
those which are polynomial, and $q_1, q_2, \ldots$ for the others.

Write $\Omega = \f_2[t_1, t_2, \ldots, q_1, q_2, \ldots ]$, a subring
of $\w_F(G)$. Then one has
$$ \w_F(G) = \Omega[p_1, p_2, \ldots ].  $$
Concretely, we shall compute the reduced Grobner basis for $\ker \pi$,
and extract from it a minimal set of generators $R_1, R_2, \ldots$ for
this ideal. The variables not showing up in any $R_k$ are the
polynomial variables (note that some of the $t_i$'s may well be
polynomial, too).

\begin{ex}\label{ex-gp16-6-1}
Throughout this section, we shall follow the example of the group $G$
already considered in proposition \ref{prop-gp16-6} and example
\ref{ex-gp16-6-replog}. The algorithm is particularly simple in this
case, yet it seems to illustrate most of the features of the general
case.

A presentation of $H^*(G)$ is as a quotient of the graded polynomial
ring $\f_2[z, y, x, w]$ with $|z| = |y| = 1$, $|x| = 3$ and $|w| =
4$. The relations are then:
$$ z^2 = 0, \qquad zy^2=0,  $$
$$ zx = 0, \qquad x^2 = 0.$$
(These form a Grobner basis.)

On the other hand, as already mentioned, we find that $\w_F(G)$ is a
quotient of the polynomial ring $\f_2[w_1(r_2), w_1(r_3), w_4(r_8)]$
where the subscript gives the degree, for some representations $r_2$,
$r_3$ and $r_8$. The relations are:
$$ w_1(r_2)^2 + w_1(r_3)^2=0, \quad w_1(r_2)^3 + w_1(r_2)^2w_1(r_3) = 0. $$
Again these form a Grobner basis.

The variable $w_4(r_8)$ is polynomial; there is no variable
corresponding to the $q_i$'s in this case.
\end{ex}

\subsubsection{Admissible maps; equivalent maps.}

An {\em admissible map} $f: \w_F(G)\to H^*(G) $ is one which is an
isomorphism in degree $1$ and which turns $H^*(G)$ into a finitely
generated $\w_F(G)$-module.

When $A$ is a graded $\f_2$-algebra, we let $A^{>0}$ denote the ideal of
elements of positive degree. If $f: A\to B$ is any map of graded
algebras, we write $\ideal{f}$ for the ideal of $B$ generated by
$f(A^{>0})$. When $A$ and $B$ are connected, then $f$ is surjective if
and only if $\ideal{f} = B^{>0}$. Also, $B$ is a finitely generated
$A$-module if and only if $B/\ideal{f}$ is finite dimensional over
$\f_2$. 

Two maps $f$ and $g$ from $\w_F(G)$ to $H^*(G)$ are said to be {\em
  equivalent} when they have the same kernel, and when $\ideal{f} =
\ideal{g}$. This defines an equivalence relation on the set of all
maps from $\w_F(G)$ to $H^*(G)$.

\subsection{Construction of certain maps $\w_F(G) \to H^*(G)$.}

The main idea is to construct all maps from $\w_F(G)$ to $H^*(G)$,
then reject those which are not admissible, then reject more maps
using finer criteria, and finally hope that the remaining maps are all
equivalent. However, we cannot {\em quite} follow this programme, for
the computation of {\em all} maps between these two rings would simply
take too much time. Careful precautions will allow us to reduce the
number of computations by many orders of magnitude. Some work will be
needed to prove that we get a correct answer nonetheless.

In this section, $f$ is a homomorphism $\w_F(G)\to H^*(G)$ which we
gradually build by specifying the values $f(t_i)$, then $f(q_i)$, and
then $f(p_i)$, step by step.

\subsubsection{Step 1 : setting the degree $1$ variables.}

We start by listing all the possible values for the various $f(t_i)$,
that is, we list all choices of $f(t_1), f(t_2) , \ldots$ such that
\begin{enumerate}
\item $f$ is an isomorphism in degree $1$,
\item the relations $R_k$ involving the elements $t_i$'s only are
  ``satisfied'', that is, map to $0$ under $f$. 
\end{enumerate}
We do this by simply exhausting all elements in degree $1$ in
$H^*(G)$, though we use the following trick in order to save time in
the sequel. Whenever we have two possible choices $f$ and $f'$, that
is whenever we have $a_1=f(t_1), a_2 = f(t_2), \ldots$ on the one hand
and $b_1= f'(t_1), b_2 = f'(t_2), \ldots$ on the other hand such that
both conditions are satisfied, we compare them thus: we check whether
the map $\alpha:H^*(G)\to H^*(G)$ sending $a_i$ to $b_i$ and all other
variables in $H^*(G)$ to themselves is well-defined. If so, it is an
automorphism of $H^*(G)$ such that $f'= \alpha\circ f $. Clearly in
this case, continuing the process with $f$ or $f'$ is immaterial for
what follows.

So we keep only one map out of the pair $(f, f')$. When this is over,
we have a set of partially defined maps; for each one, we move to the
next step. We keep writing $f$ for a particular choice.

\begin{ex}\label{ex-gp16-6-2}
Resuming example \ref{ex-gp16-6-1}, we have only one possibility for
$f$ after Step 1, in this case, namely:
$$ f(w_1(r_2)) = z + y, \qquad f(w_1(r_3)) = y.  $$
One could have exchanged the roles of $w_1(r_2)$ and $w_1(r_3)$, but
then the automorphism $\alpha$ of $H^*(G)$ which sends $y$ to $y+z$
and all other variables to themselves would bridge the two options. So
we have indeed a single $f$ that we take to the next step.

In passing note that, from the above, we see that there is more
symmetry in choosing $w_1(r_2)$ and $w_1(r_3)$, rather than $y$ and
$z$, as the generators in degree $1$.
\end{ex}

\subsubsection{Step 2 : setting the value of $f(q_i)$.}

We wish to continue in the same fashion, and find all possible values
for $f(q_i)$. Here ``possible values'' means that the remaining
relations $R_k$, not yet considered in step 1, must be
satisfied. Again we proceed by exhaustion, but a simple observation
can save us a spectacular amount of time.

Instead of defining all $f(q_i)$'s and then check whether the
relations are satisfied, we proceed one relation at a time (of
course). Given a relation $R_k$ involving $q_{i_1}, \ldots, q_{i_n}$ as
well as degree $1$ variables, we find all values for $f(q_{i_1}),
\ldots, f(q_{i_n})$ such that $f(R_k) = 0$. Then we move to the next
relation. 

{\em However, the order in which we consider the relations is
  crucial.}  Indeed, suppose that $q_{i_j}$ has degree $d_{i_j}$, and
that $H^*(G)$ is of dimension $c_{i_j}$ in degree $d_{i_j}$. Then if
$c=\sum_j c_{i_j}$, we have $2^c$ possibilities for the values of the
variables $q_{i_j}$. This number $2^c$ we call the {\em weight} of
$R_k$. We start our investigation with the relation of lowest
weight. Then, having made a choice for $f(q_{i_j})$, we recompute the
weights of the other relations, which have decreased because there are
now fewer choices to make. We proceed with the lowest weight relation
remaining, and so on.

Looking for the possibilities in this order rather than a random order
can reduce the computing time from hours to minutes.

Of course it may happen, for a given $f$ resulting from Step 1, that
there is no way of completing Step 2. However, the existence of the
map $a$ guarantees that at least one choice can pass both steps. Let
$f$ be such a map, defined on $\Omega$.

\subsubsection{Step 3 : Setting the value of $f(p_i)$.}

When we come to the polynomial variables, {\em any} value for $f(p_i)$
gives a well defined homomorphism $f$. However, in practical terms,
this means that the number of homomorphisms that we end up with is
simply too large: finishing the algorithm would take far too much
time. Instead we use the following simplification, which slightly
increases the chances of failure of the algorithm but greatly improves
the speed. (Although as we point out later, if one is particularly
interested in a single group $G$ and is willing to wait long enough,
it may be best not to use this trick).

Let $R$ denote the quotient of $H^*(G)$ by the ideal generated by
$f(\Omega^{>0})$ (in the notation above this is $R=H^*(G) / \ideal{f}$
if one keeps in mind that $f$ is only defined on $\Omega$ so far). We
extend the composition $\bar f: \Omega \to H^*(G) \to R$ to a map
$\bar f: \w_F(G) \to R$ by choosing $\bar f(p_i)$ arbitrarily
(but of the right homogeneous degree, of course). The point being that
$R$ is much smaller than $H^*(G)$ and there are relatively few choices
for $\bar f$.

Then we pick an arbitrary lift for $\bar f$, giving finally a map $f:
\w_F(G) \to H^*(G)$. Note that any two lifts $f$ and $f'$ have
$\ideal{f}=\ideal{f'}$. In particular, the finite generation of
$H^*(G)$ as a module over $\w_F(G)$ via the map $f$ does not depend on
the choice of lift.

\begin{ex}\label{ex-gp16-6-3}
We continue with the $f$ of example \ref{ex-gp16-6-2}. The dimension
of $H^4(G)$ is $3$, with a basis given by $y^4$, $yx$ and $w$ for
example. So we have $2^3=8$ choices for $f(w_4(r_8))$. However, the
ring $R$ is generated by (the images of) $x$ and $w$, so that it is
$1$-dimensional in degree $4$ (with $w$ the only nonzero element in
this degree), which leaves only $2$ choices for $\bar f$. We end up
with two possibilities for $f$: we may either send $w_4(r_8)$ to $0$,
or to $w$ (in either case, any other lift would do, but we keep only
one).

This is a toy example of course, and the computer could well exhaust
all $8$ possibilities. However, we point out that dividing the number
of subsequent computations by 4 is quite satisfactory, and such
reductions become inevitable if one wishes to deal with bigger groups
(the sensitivity being exponential).
\end{ex}

\subsection{Tests \& Conclusions}

We have now a certain finite set $S$ of maps $\w_F(G)\to H^*(G)$. We
shall now run a series of tests on the maps in $S$, leading either to
definite conclusions regarding the map $a$, or to the decision to give
up on the computation.

At this point it is not clear whether $a\in S$, or even whether $a$ is
equivalent to a map in $S$. However, there is certainly a map $f$ in
$S$ which agrees with $a$ on $\Omega$ (perhaps with the degree $1$
variables in $H^*(G)$ relabelled, cf Step 1), and with $\ideal{a} =
\ideal{f}$. 

\subsubsection{Test 1 : finite generation of $H^*(G) $.}
We reject all the maps in $S$ which are not admissible, ie those which
do not turn $H^*(G)$ into a finitely generated $\w_F(G)$-module. There
remains a smaller set $S'$. The last remark shows that $S'$ is not
empty.

\begin{ex}\label{ex-gp16-6-4}
We continue from example \ref{ex-gp16-6-3}. Out of our two maps in
$S$, only one is admissible, namely that with $f(w_4(r_8)) = w$. Thus
we keep only this $f$.
\end{ex}

\subsubsection{Test 2 : polynomial variables.}
Let $f\in S'$. We check which of the given generators for $H^*(G)$
are non zero in $H^*(G) / \ideal{f}$; for simplicity, say these are
numbered $g_1, \ldots, g_m$. We adjoin polynomial variables with the
same name to $\w_F(G)$, thus obtaining
$$ \w_F(G)^+ = \w_F(G)[g_1, \ldots, g_m].  $$
The map $f$ has an obvious extension to $\w_F(G)^+$ which we call
$f^+$; it is surjective.

We then compute the reduced Grobner basis of $\ker(f^+)$. If the
polynomial variables $p_i$ show up in this Grobner basis, that is if
the $p_i$'s are not polynomial anymore in $\w_F(G)^+ / \ker(f^+)$ with
respect to the obvious presentation, {\em we give up on the
  computation altogether}. Otherwise, if all maps in $S'$ pass this
test, we move to Test 3 with $S'$ unchanged.

We shall give below a heuristic explanation according to which it is
reasonable to expect that Test 2 is often completed succesfully. Our
interest in this test comes from the following lemma.

\begin{lem}
Let $f\in S'$ satisfy the test above, and let $g$ be a map $\w_F(G)
\to H^*(G)$ obtained by a different choice of lift in Step 3. Then $f$
and $g$ are equivalent. Also, $f^+$ and $g^+$ are equivalent.
\end{lem}

\begin{proof}
We have pointed out that $\ideal{f}= \ideal{g}$ always, so we need to
show that $\ker f= \ker g$. 

Assume that the $p_i$'s are ordered by degree, that is assume that
$|p_1| \le |p_2| \le \cdots$. Write $\Omega^+ = \Omega [g_1, \ldots, g_m]$
so that $\w_F(G)^+ = \Omega ^+[p_1, p_2, \ldots ]$. Note that $f^+$
and $g^+$ are both defined on this ring $\w_F(G)^+$, and are both
surjective. We define an automorphism $\alpha$ of $\w_F(G)^+$ such
that $f^+ = g^+\circ \alpha$.

Indeed, since $g^+$ is surjective, and from the definitions, we see
that
$$ f^+(p_i) = f(p_i) = g(p_i) + g^+(\omega_i) = g^+(p_i + \omega_i)   $$
for some $\omega_i \in \Omega^+[p_j : |p_j| < |p_i| ]$. So we may define
$\alpha$ by requiring it to be the identity on $\Omega^+$, and to send
$p_i$ to $p_i + \omega_i$. In order to see that $\alpha$ is an
isomorphism, one may for example show by induction on $i$ that
$\omega_i$, and thus $p_i$, is in the image of $\alpha$; therefore
$\alpha$ is surjective and is an isomorphism as a result.

It is now easy to conclude. Let $b_1, b_2, \ldots $ be the reduced
Grobner basis for $\ker(f^+)$. By choosing the term order carefully,
we can arrange things so that the $b_i$'s not involving the variables
$g_1, \ldots, g_m$ constitute the reduced Grobner basis for $\ker f$;
say these are $b_1, \ldots, b_r$. Now, since $f$, or rather $f^+$,
passes Test 2, then the elements $b_1, \ldots, b_r$ do not involve the
$p_i$'s, either. It follows that $\alpha(b_i) = b_i$ and that $\ker f
\subset \ker g$.

From the relation $g^+ = f^+ \circ \alpha^{-1}$, it is clear that $g$
passes Test 2 as well, so we may reverse the roles and obtain $\ker g
\subset \ker f$.

Proving that $f^+$ and $g^+$ are equivalent is a similar, but easier,
matter. 
\end{proof}

Assuming that all maps in $S'$ have passed the test, we can move on to
Test 3 knowing that $a$ is equivalent to some map in $S'$.

\begin{ex}\label{ex-gp16-6-5}
We continue from example \ref{ex-gp16-6-4}. There is only one $f$ to
deal with. In this case $H^*(G)/\ideal{f}$ is generated by $x$ only
as an algebra, so we adjoin a variable $x$ to $\w_F(G)$, obtaining
$\w_F(G)^+$ which is generated by $w_1(r_2)$, $w_1(r_3)$, $w_4(r_8)$
and $x$. We extend $f$ to this ring by setting $f^+(x)=x$.

A Grobner basis for $\ker (f^+)$ is then
$$ w_1(r_2)^2 + w_1(r_3)^2, \qquad w_1(r_2)^2w_1(r_3) + w_1(r_2)w_1(r_3)^2,$$
$$ w_1(r_2)x + w_1(r_3)x, \qquad x^2.  $$
These do not involve $w_4(r_8)$. Test 2 is successful.

Note that, since we have only one map in $S'$, there is no need to
perform Test 3 and Test 4, which we describe now in the general case.
\end{ex}

\subsubsection{Test 3 : Steenrod operations.}
We reject all maps in $S'$ whose kernel is not stable under the
Steenrod operations. There remains a smaller set $S''$. Since $a$ is a
map of unstable algebras, $S''$ is not empty.

\subsubsection{Test 4 : restrictions to elementary abelian subgroups.}
When $E$ is an elementary abelian $2$-group, then $H^*(E)$ is
completely understood, including Stiefel-Whitney classes. One way to
state this is to say that $a_E : \w_F(E)\to H^*(E)$ is an isomorphism,
and that (as always) the map $A^*_E \to \w_F(E)$ is explicitly
described. 

We exploit this to setup our final test. The map $a:\w_F(G) \to
H^*(G)$ can be composed with the restriction $H^*(G) \to H^*(E)$ for
any elementry abelian subgroup $E$ of $G$, thus giving a map $\w_F(G)
\to H^*(E)$ which we understand fully: to determine the image of
$w_j(r_i)$, decompose $r_i$ as a sum of irreducible, real
representations of $E$, and compute the Stiefel-Whitney of this sum
using the usual formula; then use the map $A^*_E \to H^*(E)$ to
express the result in terms of your favorite choice of generators for
$H^*(E)$.

Our test is the following. If the generators for $\ker f$ do not map
to $0$ under the restriction maps $\w_F(G) \to H^*(E)$, for $E$
running among the maximal elementary abelian subgroups, then we reject
$f$ from $S''$. We obtain in this way a smaller, nonempty set $S'''$.

\subsubsection{Conclusion.}
If the maps in $S'''$ are not all equivalent, the computation has
failed. If they are, we compute for each $f\in S'''$ the map $f^+$ has
above; note that all these are defined on the same ring
$\w_F(G)^+$. If the maps $f^+$ are not all equivalent, the
computation has failed. Otherwise, we claim that we have succeeded, in
a sense which we make precise now.

Pick an $f$ in $S'''$. Then $f$ is equivalent to $a$. Moreover $f^+$
and $a^+$ are defined on the same ring $\w_F(G)^+$, are both
surjective, and are equivalent. Thus, we know the kernel of the
surjective map
$$ a^+ : \w_F(G)[g_1, \ldots, g_m] \to H^*(G).  $$
This is the desired presentation of $H^*(G)$.

\begin{ex}
We conclude example \ref{ex-gp16-6-5}, and the proof of proposition
\ref{prop-gp16-6} at the same time. There being only one candidate in
$S'$, Test 3, Test 4 and the final check are all redundant. Note that
the map $f$ is not necessarily equal to $a$: in Step 1 we had two
choices, and in Step 3 we had four, so we can write down 8 maps from
$\w_F(G)$ to $H^*(G)$, one of which will be $a$. However, these are
all equivalent, and their extensions to $\w_F(G)^+$ are also all
equivalent. In the end, we know the kernel of $a^+$, as it was given
in the previous example. Thus proposition \ref{prop-gp16-6} holds (the
information on Steenrod operations follows from Wu's formula).
\end{ex}

\subsection{Comments}

All of the comments below will have something to do with the trick
used in Step 3 and its validation in Test 2.

\subsubsection{A variant.}
There is an evident variation that we may want to try: namely, in Step
3, drop the ring $R$ and the choice of lifts altogether, and simply
gather all possible homomorphisms by listing all possible values for
the polynomial variables. Then Test 2 becomes irrelevant.

As already pointed out, this will often lead to a number of elements
in the set $S$ which is impossible to manage: each homomorphism in $S$
will need to have its kernel computed, and this uses Buchberger's
algorithm for Grobner basis, a time-consuming process of exponential
complexity. However, in very particular cases, it may still be best to
go down this road anyway.

\begin{ex}
Consider the group number $12$ (in the GAP library) of order $64$; it can be described as $(\z/4\rtimes \z/8) \rtimes \z/2$. The algorithm above produces about $60$ homomorphisms
after Step 3, and Test 2 fails. However, the variant algorithm
produces 384 homomorphisms, which are all surjective and fall into $9$
equivalence classes. They all pass Test 3. Fortunately only one of
them passes Test 4, and the computation is complete.

Similarly, we may look at the group number $87$ of order $64$, a group of type $\z/2 \times ((\z/8\times \z/2)\rtimes \z/2 )$. The
normal algorithm yields about 800 homomorphisms, and Test 2 fails. It
is still possible to use the variant, even though there are now 24,576
homomorphisms to deal with. They are all surjective, fall into $5$
equivalence classes, only one of which passes Test 3. The computation
is complete, and takes about 30 minutes on an average
computer. Clearly, we cannot let the complexity gain an extra order of
magnitude.
\end{ex}

\subsubsection{The success of Test 2.}
There are above 100 groups for which our computations are successful;
only 4 of them have required the lengthy alternative algorithm. On the
other hand, in the vast majority of cases, when the computation fails
it does so for reasons other than Test 2. This means that the test is
often passed, and indeed it was our hope that it should be so.

A loose explanation is as follows. The ring $\w_F(G)$ is sufficiently
fine that the kernel of $a$ is relatively small; in particular, if a
variable is polynomial in $\w_F(G)$, it is unlikely that its image
should not be polynomial in $H^*(G)$. So the map $a$ itself should
pass Test 2. Now, this tells us something about the size of $H^*(G)$
relative to that of $\w_F(G)$, and if another homomorphism $f: \w_F(G)
\to H^*(G)$ were to fail Test 2, that is, were to have a polynomial
variable showing up in its kernel, it is likely to have an image which
is too small, and thus Test 1 will reject it. Otherwise,  $f$
would have to have a much bigger image on $\Omega$ than $a$ does,
which again is unlikely.

The examples above show that ``unlikely'' does not mean
``impossible''. We also note that a refinement of $\w_F(G)$, which one
could obtain by thinking of more relations to throw in, would increase
the chances of our algorithm.

\subsubsection{The order of the tests.}
It is tempting to run the straightforward Test 3 and Test 4 first, and
thus have a smaller set of homomorphisms on which to try the more
dubious Test 2. However, this cannot be done. Indeed, a map could well
fail Test 3, say, whereas another choice of lift in Step 3 would give
a map that passes it.

%% file: results.tex
\section{Experimental results}\label{sec-results}

We shall first comment on the practicalities on the computations, and
then on the mathematics.

\subsection{Practicalities}

\subsubsection{The programs.}
The first task is to gather information on the characters of the group
$G$, and on the sizes of the conjugacy classes. From this, one can
compute scalar products between characters, and thus express tensor
products and exterior powers in terms of the irreducible
representations. One also finds out what the real characters are. All
this is done with the help of the GAP computer package.

The bulk of our project, comprising more than 99\% of the code, is a
C++ program which computes a presentation for $\w_F(G)$ and then goes
through the algorithm just presented. There are about $18,000$
lines of C++ code in standard presentation, to which one must add
about $5,000$ lines of comments (by comparison, the \LaTeX ~source for
the present article has just above $2,000$ lines).

It is also necessary to get the information on $H^*(G)$ from Carlson's
webpage, which is presented there as a Magma file. In order to
download all the necessary files automatically and translate them into
C++, we have used the Python programming language. Incidentally,
Python was also used to produce the various HTML files containing the
results.

It has been very convenient to use the SAGE computer package, which
allows the smooth blending of GAP, Python, and C++.

\subsubsection{The computing time.}
All computations were performed on the {\tt irmasrv3} server at
the university of Strasbourg. This machine has 12 CPUs, which was
extremely handy to run the various calculations in parallel. Each CPU
though has the power of a standard, personal machine.

The preliminaries, before the algorithm of section \ref{sec-algo}
starts, take little time. It may happen that the computation of
universal polynomials, used in the formulae for tensor products for
example as in section \ref{sec-formal}, take several minutes.

The main algorithm can in many cases be completed in a few seconds;
sometimes it can take above 20 minutes (group 87 of order 64); or it
can take several hours (for example for $Q_8\times (\z/2)^3$, for
which it is of course preferable to use the Kunneth formula).

Also, occasionally, the algorithm seems to take so long that we
have interrupted it and given up on the computation. The reader may be
surprised to learn that it is mostly innocent-looking Step 2 which is
particularly time-consuming. {\em This is in fact the most common
  cause of failure of the algorithm}, much more frequently encountered
than a failure after Test 2 or at the very end when there are more
than one equivalence class.

\subsection{Mathematical results}

\subsubsection{Success.} 
As announced in the introduction, we have focused on the groups of
order dividing $64$. The computation was successful for the 5 groups
of order 8, for 13 of the 14 groups of order 16, for 28 of the 51
groups of order 32, and for 61 of the 267 groups of order 64 (a total
of 107 groups).

Note that the method is not well-behaved with respect to products:
even if we can successfully run the computation for both $G$ and $H$,
it may still fail for $G\times H$ (because the complexity explodes).

\subsubsection{Cohomology rings generated by Stiefel-Whitney classes.}
Among our succesful computations, only 13 groups have been found to
have a cohomology which is {\em not} generated by Stiefel-Whitney
classes. Of course one may argue that the algorithm is more likely to
terminate without incident when the cohomology is generated by such
classes (and there are no maps $f^+$ to consider at all). However we
have pointed out that the main cause of failure is the excessive time
needed by the calculations, and so we find it reasonable to conclude
that ``most'' groups have $\w(G) = H^*(G)$. 


\subsubsection{A curiosity: groups with isomorphic cohomologies.} Let
$G$ be the group of order 32 whose Hall-Senior number is 21; its
number in the GAP library is 12, and it can be described as a
semidirect product $\z/4 \rtimes \z/8$. On the other hand, let $G'$ be
the group of order 32 whose Hall-Senior number is 29; its number in
GAP is 14 and it is also a semidirect product, this time $\z/8 \rtimes
\z/4$. 

Then $H^*(G)$ and $H^*(G')$ are isomorphic rings. What is more, our
computations show that {\em they are isomorphic as unstable algebras},
that is, there is an isomorphism between them which commutes with the
Steenrod operations.

This implies classically (\cite{lannes}, Prop. 3.1.5.2) that, for any elementary abelian $2$-group $E$, there exists a bijection $Rep(E, G) = Rep(E, G')$ (the set $Rep(A,B)$ consists of all group homomorphisms from $A$ into $B$ up to conjugacies in $B$).

%% file: milnor.tex
\section{Application to algebraic cycles}\label{sec-chow}

\subsection{Algebraic cycles in the cohomology}

\subsubsection{The Chow ring.}

For any algebraic group $G$ over $\c$, for example a finite group, the
classifying space $BG$ can be approximated by algebraic varieties, in
such a way that there is a well-defined {\em Chow ring} $CH^*BG$. As
the notation suggests, everything works as if $BG$ were an algebraic
variety itself, and $CH^*BG$ is to be thought of as generated by the
subvarieties of $BG$. For details see \cite{totaro}.

There is a {\em cycle map}
$$ cl: CH^*BG \to H^{2*}(BG, \z),$$
whose image we denote by $\ch(G)$. Cohomology classes in $\ch(G)$ are
usually said to be supported by algebraic varieties. Our aim is to
compute $\ch(G)$ for as many groups $G$ as possible, using our results
on Stiefel-Whitney classes. More precisely, we will obtain information
on the composition
$$ CH^* BG \otimes_\z \f_2 \to H^{2*}(BG, \z)\otimes_\z \f_2 \to
H^{2*}(BG, \f_2).  $$
This map we still denote by $cl$, and its image by $\ch(G)$. Also,
$CH^*BG$ will stand for $CH^*BG \otimes_\z \f_2$ from now on, unless
we repeat the reduction mod 2 for emphasis. Recall that we write
$H^*(G)$ for the mod $2$ cohomology of $G$.

\subsubsection{A lower bound.}
If $V$ is a complex representation of $G$, then it has Chern classes
$c_i(V) \in H^*(G)$, which are pulled-back from $H^*(BGL_n(\c),
\f_2)$. However, it turns out that the cycle map $cl$ is an
isomorphism for $GL_n(\c)$ (see \cite{totaro}), so we have an
identification $ H^*(BGL_n(\c), \f_2) = CH^*BGL_n(\c)$.

What is more, the cycle map $cl$ is natural in $G$. It follows that
the Chern classes $c_i(V)$ ``come from the Chow ring'', ie are in the
image of the cycle map for $G$. In symbols,
$$ \C(G) \subset \ch(G).  $$

\subsubsection{An upper bound.}
The Steenrod algebra acts on the ring $CH^*BG\otimes_\z \f_2$, and the
cycle map commutes with the operations $Sq^k$: for this see
\cite{brosnan}. However $Sq^1$ acts trivially. Note that $CH^*BG$ is
often seen as a graded ring concentrated in even degrees, with $CH^nBG$
in degree $2n$; with this convention, $Sq^k$ raises the degrees by
$k$, so if $k$ is odd then $Sq^k$ must be zero on the Chow ring indeed.

Let $\ideal{Sq^1}$ be the two-sided ideal generated by $Sq^1$ in the
mod 2 Steenrod algebra. We see that $\ideal{Sq^1}$ acts as $0$ on the
Chow ring of any variety; as a result, {\em any class in $\ch(G)$ is
  killed by $\ideal{Sq^1}$.}

It is traditional to write $\tob H^*(G)$ for the subring of $H^*(G)$
of all those even-degree classes which are killed by
$\ideal{Sq^1}$. This is the largest unstable submodule of $H^*(G)$
which is concentrated in even degrees. With this notation one has: 
$$ \ch(G) \subset \tob H^*(G) .  $$

\subsection{Computations}

Our strategy is pretty simple-minded: we shall compute $\C(G)$ and
$\tob H^*(G)$, and hope that they coincide. In such cases (which are
quite common, as we shall see), these two subrings also coincide with
$\ch(G)$. 

The ring $\C(G)$ is trivial to describe for those groups $G$ for which
our previous computations were successful: indeed we have explained in
\S \ref{sec-chern} how to express Chern classes in terms of
Stiefel-Whitney classes, and we have a full understanding of the map
$A^*_G \to \w(G)$.

As for $\tob H^*(G)$, we need a couple of results before we start.

\subsubsection{Milnor derivations.}

Define $Q_0 = Sq^1$ and
$$ Q_{n+1} = Sq^{2^{n+1}}Q_n + Q_n Sq^{2^{n+1}}.  $$
Then each $Q_i$ acts as a derivation on any unstable algebra, and is
called the $i$-th {\em Milnor derivation} (see
\cite{milnor-der}). They all commute with each other.

\begin{lem}\label{lem-tob-milnor-der}
If $A$ is an unstable algebra and $x\in A$ has even degree, then $x$
belongs to $\tob A$ if and only if $Q_i(x) = 0$ for all $i\ge 0$.
\end{lem}

\begin{proof}
In fact, let $A'$ denote the algebra of all elements (of even degree
or not) killed by each $Q_i$, and let $A''$ denote the algebra of all
elements (again, of arbitrary degree) killed by $\ideal{Sq^1}$. We
prove that $A'=A''$.

Since the Milnor derivations are clearly in $\ideal{Sq^1}$, we
certainly have $A''\subset A'$. On the other hand, $Q_0= Sq^1$, so it
suffices to show that $A'$ is stable under the Steenrod operations to
get the inclusion $A'\subset A''$.

This follows from \cite{milnor-der}, theorem 4a, from which we extract
just one formula:
$$ Q_k Sq^r = \sum_{i=0}^\infty Sq^{r - 2^k (2^{i+1}-1)}Q_{k+i}. $$
This is really a finite sum, with the convention that $Sq^a=0$ when
$a<0$. Clearly this proves the claim.
\end{proof}

\subsubsection{The kernel of a derivation.}
If $A$ is an algebra over a field $k$, and $$d: A\to A$$ is a
derivation, how are we to compute generators for the algebra $\ker d$
? Here is the simple method which we have used. 

We assume that we have
a subalgebra $B$ of $A$ such that:
\begin{itemize}
\item $d$ vanishes on $B$. Thus $d$ is $B$-linear when $A$ is viewed as
  a $B$-module.
\item there is a presentation $r_A: \tilde A \to A$, resp $r_B: \tilde
  B \to B$, where $\tilde A$, resp $\tilde B$, is a polynomial
  ring. These are compatible in the sense that there is a commutative
  diagram
$$ \begin{CD}
\tilde B     @>>>     \tilde A     \\
@V{r_B}VV              @VV{r_A}V   \\
     B       @>>>          A
\end{CD}
  $$
where the horizontal maps are inclusions.
\item $\tilde A$ is a free $\tilde  B$-module of finite rank $n$.
\end{itemize}

In the case at hand, namely finitely generated algebras over $\f_2$,
it is easy to find such a $B$, mostly because $d$ {\em vanishes on
  squares}. Assume that $A$ is presented as a quotient of $\tilde A=
\f_2[X_i]$, and write $x_i= r_A(X_i)$. If $d(x_i)=0$, put
$Y_i=X_i$; if not, put $Y_i= X_i^2$. The algebra $\tilde B= \f_2[Y_i]$
and its quotient $B= \tilde B / \ker r_A$ together satisfy the
properties given. (Note that we could simply take $Y_i=X_i^2$ for all
$i$, but this increases the rank $n$, which is not desirable in practice.)

Let us introduce some notations. We let $\tilde \varepsilon_1, \ldots,
\tilde \varepsilon_n$ be generators for $\tilde A$ as a $\tilde
B$-module. Using these we can and we will identify $\tilde A$ and
$\tilde B^n$. Put $\varepsilon_i= r_A(\tilde \varepsilon_i)$. We write
$p: \tilde B^n \to A$ for the map of $\tilde B$-modules underlying
$r_A$, and we let $\sigma_1, \ldots, \sigma_k$ be generators for $\ker
p$. (When we know generators $f_1, f_2, \ldots$ for $\ker r_A$ {\em as
  an ideal}, then the collection of all elements $\tilde \varepsilon_i
f_j$ provides a choice of such generators for $\ker p$).

We then pick a lift $\tilde d$ of $d$:
$$ \begin{CD}
\tilde B^n     @>{\tilde d}>>     \tilde B^n     \\
@V{p}VV                           @VV{p}V   \\
     A          @>{d}>>               A
\end{CD}
  $$
Let $d_i= \tilde d (\tilde \varepsilon_i) \in \tilde B^n$. Now if $x=
\sum b_i \varepsilon_i \in A$, with each $b_i \in B$, then $x$ belongs
to $\ker d$ if and only if $\sum \tilde b_i d_i \in \ker p$, where
$r_B(\tilde b_i) = b_i$. In other words this happens if and only if
there exist elements $c_i \in \tilde B$ such that
$$ \sum_{i=1}^n \tilde b_i d_i + \sum_{i=1}^k c_i \sigma_i = 0. $$
Or, to say this yet differently, the element $(\tilde b_1, \ldots,
\tilde b_n, c_1, \ldots, c_k)$ of $\tilde B^{n+k}$ belongs to the {\em
  syzygy module} of the elements $d_1, \ldots, d_n, \sigma_1, \ldots,
\sigma_k$, which all live in $\tilde B^n$. 

Now, computing generators for the syzygy module of a collection of
elements in a free module over a polynomial ring is standard
computational algebra\footnote{as is, probably, the computation of the
  kernel of a derivation. In this paragraph our goal is to justify and
  explain our own method, in particular to the benefit of readers of
  the source code. We point out that feeding our results on Steenrod operations to the appropriate sofware would require a considerable amount of work anyway; and more seriously, as we proceed recursively with $d$ playing the role of each $Q_i$ one after the other, we have been able to include a number of optimizing tricks, saving work between one computation and the next.}, see \cite{adams}. Having computing generators, we
only keep their first $n$ coordinates (ie we keep the $\tilde b_i$'s
and drop the $c_i$'s). Applying $p$ yields generators for $\ker d$.

\subsubsection{Computing in finite time.}
We are now prepared to compute $\ker Q_i$ for each $i$. The algebra
$\tob H^*(G)$ is the intersection of all those. The intersection of
the first $N+1$ kernels is easy to obtain, for it is the kernel of
$Q_N$ viewed as a derivation on the algebra recursively computed as
the intersection of the first $N$ kernels. (Recall that the Milnor
derivations commute).

However, we would like to compute only a finite number of kernels. The
next lemma follows from \ref{lem-tob-milnor-der}.

\begin{lem}
If the integer $N$ is such that the even part of
$$ \bigcap_{i=0}^N \ker Q_i  $$
is stable under the Steenrod operations, then this even part is
$\tob H^*(G)$.
\end{lem}

This lemma provides an easy test for completion. Moreover:

\begin{prop}
There exists an $N$ as in the lemma. In other words, the computation
of $\tob H^*(G)$ terminates in finite time.
\end{prop}

\begin{proof}
Let $A= H^*(G)$, and let $\Omega = \Omega_{A/\f_2}$ be the module of
$\f_2$-differentials of $A$. Then $Der(A, A) = Hom_A(\Omega, A)$.

If $A$ is generated by elements $x_1, \ldots, x_n$ as an algebra, then
$\Omega$ is generated by the elements $dx_1, \ldots, dx_n$ as an
$A$-module, and in particular it is finitely generated. Thus
$Hom_A(\Omega, A)$ injects in a free $A$-module of finite rank, and
since $A$ is Noetherian, it follows that $Der(A,A)$ is finitely
generated as an $A$-module.

Thus for all $N$ larger than some $N_0$, the derivation $Q_N$ is an
$A$-linear combination of the derivations $Q_0, \ldots, Q_{N_0}$, and
we see that $$\ker Q_N \subset \bigcap_{i=0}^{N_0} \ker Q_i.$$ Therefore
we may take $N_0$ in the lemma.
\end{proof}

\subsection{Results}

\subsubsection{Success and failure.}
We have attempted to go through the above procedure for all the groups
$G$ ``at hand'', namely all those for which the computation of
Stiefel-Whitney classes was completed and for which $H^*(G) =
\w(G)$. In such cases the action of the Steenrod operations is already
given, while for other groups more work would be needed to find out
the action of each $Sq^k$ on cohomology classes which are not
Stiefel-Whitney. So we had $107 - 13 = 94$ groups to try (see
\S\ref{sec-results}).

In principle there is nothing to prevent the calculation from reaching
its end. However in practice, it may happen that the computer runs
out of memory, or that the computation takes simply too long. We have
obtained answers in the following cases: for all $5$ groups of order
$8$, for $10$ groups of order $16$, for $17$ groups of order $32$, and
for $30$ groups of order $64$, a total of $62$ groups.

Among these, $38$ groups $G$ have $\C(G) = \ch(G) = \tob H^*(G)$ (and
of course all three are explicitly presented). It is interesting to
note that, in the remaining cases, one has at least
$$ \C(G)/\sqrt(0) = \ch(G)/\sqrt(0) = \tob H^*(G)/\sqrt(0)  $$
where $\sqrt(0)$ denotes the ideal of elements which square to $0$
({\em not} the radical). This is slightly stronger than what the
general theory predicts, which is that $\C(G) \to \tob H^*(G)$ is an
$F$-isomorphism (in other words, it is known that for any $x \in\tob
H^*(G)$, one has $x^{2^n}\in\C(G)$ for $n$ sufficiently large).






\subsubsection{A worked out example.}
Let us consider $G= (\z/2)^3$. It is easy to perform the computations
by hand (see below), but this example will serve well to illustrate
what the computer does.

We have $H^*(G)=\f_2[x, y, z]$ where $x=w_1(r_1)$, $y=w_1(r_2)$, and
$z = w_1(r_3)$ for some $1$-dimensional, real representations $r_1,
r_2$ and $r_3$. One has
$$ Sq^1 = x^2 \frac{\partial} {\partial x} +  y^2 \frac{\partial}
{\partial y} + z^2 \frac{\partial} {\partial z}.  $$
In the notations above, we take $A=H^*(G)$ and $B= \f_2[x^2, y^2,
  z^2]$. Generators for $A$ as a $B$-module are $1, x, y, z, xy, xz,
yz, xyz$ (here $A=\tilde A$ and $B= \tilde B$). 

One computes $d_2 = Sq^1 x = x^2 = (x^2, 0, 0, 0, 0, 0, 0, 0 ) \in B^8$,
then $d_5 = Sq^1 xy = x^2y + xy^2 = (0, y^2, x^2, 0, 0, 0, 0)$, and so on.

The module of syzygies for the elements $d_1, \ldots, d_8$ has 21
generators. However as an algebra, we find that
$$ \ker Sq^1 = B[a_1, a_2, a_3, a_4]  $$
where
$$ a_1=  z^2x + zx^2 \qquad  a_2= y^2x + yx^2$$
$$ a_3= y^2z + yz^2 \qquad a_4=  y^2zx + yz^2x + yzx^2.$$
The even part of $\ker Sq^1$ is not stable under the Steenrod
operations, for
$$ Sq^2(a_4) = y^4zx + yz^4x + y^2z^2x^2 + yzx^4 $$
and applying $Sq^1$ to the right hand side does not give $0$.

So we move to
$$ Q_1 = x^4 \frac{\partial} {\partial x} +  y^4 \frac{\partial}
{\partial y} + z^4 \frac{\partial} {\partial z}.  $$
Now $A=\ker Sq^1$ while $B$ stays the same. The generators for $A$ as
a $B$-module are $1$, $a_4$, $a_3$, $a_2$, $a_1$, $a_3a_4$, $a_2a_4$,
$a_2a_3$, $a_1a_4$, $a_1a_3$, $a_1a_2$, $a_2a_3a_4$, $a_1a_3a_4$,
$a_1a_2a_4$, $a_1a_2a_3$, $a_1a_2a_3a_4$.

One applies $Q_1$ to these and then computes the module of
syzygies. There are 27,730 generators for this module, and only one
(!) generator for $\ker Q_1\cap \ker Sq^1$ as an algebra, namely
$$ \ker Q_1\cap \ker Sq^1 = B[ y^4a_1 + y^2x^2a_1 + z^4a_2 +
  z^2x^2a_2 ].  $$
The even part of this algebra is $B$, on which every $Q_i$ clearly
vanishes (every element of $B$ being a square in $H^*(G)$). In
conclusion
$$ \tob H^*(G) = \f_2[x^2, y^2, z^2].  $$
Of course $w_1(r_i)^2 = c_1(r_i\otimes \c)$ ($i=1,2,3$), so $\tob
H^*(G)$ is generated by Chern classes and $\C(G)=\tob H^*(G)=\ch(G)$.

Similar results hold for $(\z/2)^n$ for any $n$, and the shortest
proof is by induction (in good cases one can use a K\"unneth formula
for $\tob H^*(G)$, see \cite{lionel}).

%% file: theory.tex
\section{Theoretical considerations}\label{sec-theory}

In this Appendix, we expose three methods that we recommend in order
to compute the Stiefel-Whitney classes. They are all ``theoretical''
methods, in that in each case there are serious diffulties arising
when one attempts to carry the method into practice. What we have described up to this point is a way to circumvent the hard work in a lot of cases.

Throughout this section, $G$ denotes a finite group.

\subsection{The Atiyah-Evens approach}

Historically, Atiyah was the first to ask for a purely algebraic
definition of the Chern (rather than Stiefel-Whitney) classes: see
\cite{atiyah}. The first answer was provided by Evens, based on his
multiplicative ``norm'', see \cite{evens}. This approach was
generalized by Fulton and MacPherson in \cite{fulton}. We base our
discussion on \cite{kozlow}.

The basic strategy is as follows. Suppose that $G$ is a
$2$-group. Then any irreducible, complex representation of $G$ is
induced from a $1$-dimensional representation of a subgroup of $G$
(see \cite{serre}). For real representations, the corresponding
statement is: any irreducible, real representation of $G$ is induced
from a representation of a subgroup $K$ of $G$ which is either $1$- or
$2$-dimensional, and in either case obtained from a homomorphism $K\to
C$, where $C$ is a finite cyclic $2$-group of roots of unity in $\c$.

Now, the cohomology of $C$ is completely understood, of course. To be
precise, when $C$ is of order $2$, it has a nontrivial real
representation of real dimension $1$, and its first Stiefel-Whitney
class is the only nonzero class in $H^1(C, \f_2)$; if $C$ has order
$\ge 4$, it has a real, irreducible representation $V$ of dimension
$2$ obtained by viewing $C$ has a group of roots of unity, and one has
$w_1(V)=0$ while $w_2(V)=c_1(V)$ is the only nonzero class in $H^2(C,
\f_2)$. The representation of $K$ considered above is the
``restriction'' of one of these, so the Stiefel-Whitney classes may be
computed by pulling back the classes in $H^*(C, \f_2)$ to
$H^*(K,\f_2)$.

The difficult part is to obtain a formula for the Stiefel-Whitney (or
Chern) classes of an induced representation, given the corresponding
classes in the cohomology of the subgroup $K$. As noted above, Evens
was the first to provide such a formula, valid only for Chern classes,
while Fulton and MacPherson gave a very general statement. See also
Evens and D. Kahn \cite{evenskahn}, B. Kahn \cite{kahn}, and Kozlowski
\cite{kozlow2}. We give the Fulton-MacPherson formula in the case when
$K$ has index $2$ in $G$:
$$ w(Ind(r)) = N(w(r)) + \sum_{d=0}^{e-1}[(1+\mu)^{e-d}+1]N(w_d(r))  $$
There are quite a few notations to explain. Here $r$ is a
representation of $K$, and $e$ is its real dimension; $Ind(r)$ is the
representation of $G$ induced by $r$. The notation $w(\rho)$ stands
for the total Stiefel-Whitney class of $\rho$:
$$ w(\rho)= 1 + w_1(\rho) + w_2(\rho)+ \cdots  $$
which is a non-homogeneous element in the cohomology of the group of
which $\rho$ is a representation. Further, $N$ is the Evens norm from
$K$ to $G$: recall that this is a map $H^*(K, \f_2) \to H^*(G, \f_2)$
which is neither additive nor degree-preserving, but it is
multiplicative; moreover $N$ can be computed algebraically, see
\cite{carlson}. Finally, $\mu$ is the class in $H^1(G, \f_2)$
determined by the homomorphism $G \to G/K=\z/2\z$.

There is also a formula when $K$ has index greater than $2$, but it is
much more complicated. It seems easier, in this case, to consider a
series of subgroups $K\subset K_1 \subset K_2 \subset \cdots \subset
K_n=G$, each of index $2$ in the next, and to use the formula repeatedly.

\subsubsection{Pros and cons.}
The advantage of this method is its relative simplicity (compare
below). However, in practice, there is a serious obstacle to overcome:
namely, following the method may lead one to compute the cohomology of
very many subgroups of $G$, together with the Evens norm in each
case. When $G$ is large, it is an understatement to say that the
computation is discouragingly long.

\subsection{The Thom construction}

We shall now describe a discrete, or combinatorial, version of the
Thom construction, which also allows the computation of
Stiefel-Whitney classes.

\subsubsection{The topological side.}
We recall the following well-known facts (see \cite{milnor}). Let $X$
be any topological space, and let $E$ be a real vector bundle over
$X$. Also, let $E_0$ denote the complement of the zero-section in
$E$. Then there is a {\em Thom isomorphism}
$$ T: H^d(X, \f_2) \stackrel{\simeq}{\longrightarrow} H^{n+d}(E, E_0; \f_2)  $$
where $n$ is the rank of $E$. When $X$ is connected, so that there is
a unique nonzero element $1\in H^0(X, \f_2)$, we call $T(1)\in H^n(E,
E_0;\f_2)$ the {\em Thom class of $E$}. As it turns out, the Thom
isomorphism is given by cup-multiplication with $T(1)$.

Consider then the element $Sq^iT(1)$. It corresponds, via the Thom
isomorphism, to a class in $H^i(X, \f_2)$. This class is $w_i(E)$
(indeed, this is a possible {\em definition} of the Stiefel-Whitney
classes). 

If now $X=BG$ and $V$ is a real representation of $G$, we may consider
the universal $G$-principal bundle $EG \to BG$ and form from it the
vector bundle $(EG \times V) \to BG$. Call it $E$. Then
$w_i(V)=w_i(E)$.

\subsubsection{A finite CW complex acted on by $G$.}
Let us start with a real vector space $V$ of dimension $n$ and a set
of points $A=\{a_1, \ldots, a_m\}$ whose affine span is all of $V$. We
let $\Delta$ denote the convex hull of $A$, a {\em polyhedron} in $V$.

A {\em boundary plane} of $\Delta$ will mean a hyperplane of $V$ which
intersects the topological boundary of $\Delta$ but not its
interior. It follows from the Hahn-Banach theorem that the boundary of
$\Delta$ is the union of all the boundary planes.

If $P$ is a boundary plane, then $\Delta\cap P$ is the convex hull of
a subset $A'\subset A$. When the affine span of $A'$ is the whole of
$P$, we call it a {\em supporting plane}, and $\Delta\cap P$ is called
a {\em face} of $\Delta$. It is not hard to show that the boundary of
$\Delta$ is the union of the faces.

Since each $\Delta\cap P$ is itself a polyhedron in $P$, one can
define inductively the $k$-faces of $\Delta$ (which are the
$k-1$-faces of the faces). The $n$-faces are the {\em vertices} of
$\Delta$, and they form a subset of $A$. This may be a strict subset
of $A$ in general (say if $a_2$ is the middle of the segment from
$a_1$ to $a_3$), but if we have chosen a Euclidean metric on $V$ and
chosen the $a_i$'s on the unit sphere, then none of them can belong to
the interior of any $k$-face, so that the set of vertices of $\Delta$
is precisely $A$.

Topologically, $\Delta$ is an $n$-cell, and its boundary is an
$n-1$-sphere. It follows that the above decomposition into $k$-faces
yields a decomposition of $\Delta$ as a CW-complex, each $k$-face
giving an $n-k$-cell. We let $\Delta_*$ denote the corresponding mod
$2$ cell complex. Since we need not worry about the signs here, the
boundary of a $k$-face is, quite simply, the sum of its
faces. Similarly, we shall write $Bd(\Delta)$ for the boundary of
$\Delta$, and $Bd(\Delta)_*$ for the corresponding complex. The
discussion above is meant to show explicitly how to compute the above
cell complexes in finite time.

The case of interest to us is that of an irreducible real
representation $V$ of the group $G$, and $A= \{ g_ov, g_1v, \ldots,
g_nv \}$ for some nonzero $v\in V$, where the elements of $G$ have
been written $g_0, \ldots, g_n$. The vector span of $A$ is $V$ since
$V$ is irreducible. Assuming that $V$ is nontrivial, we see that the
invariant element $g_0v + \cdots + g_nv$ is $0$, so that the
barycenter of $\Delta$ is the origin in $V$, and it follows that the
{\em affine} span of $A$ is $V$. We may assume that there is a
$G$-equivariant Euclidean metric on $V$, so that the vertices of
$\Delta$ are the points $g_iv$.

There is an action of $G$ on $A$ and also on $\Delta$ and $\Delta_*$;
we see the latter as a complex of $\f_2[G]$-modules. Note that there
is a homeomorphism from $\Delta$ to the unit ball in $V$, carrying its
boundary to the unit sphere, defined by sending each ray emanating
from the origin to a corresponding ray in the same direction, with an
appropriate rescaling. Since the action of $G$ on $V$ is linear, this
homeomorphism is $G$-equivariant.

\subsubsection{Resolutions.} We continue with the notations for $V$
and $\Delta$. If $P_*$ is any projective resolution of
$\f_2$ as an $\f_2[G]$-module, we define $\bar P_*$ to be the cokernel
of 
$$ P_*\otimes_{\f_2} Bd(\Delta)_* \to  P_*\otimes_{\f_2} \Delta_*.  $$
Any chain homotopy between $P_*$ and $Q_*$ yields a chain homotopy
between $\bar P_*$ and $\bar Q_*$. Thus we are free to pick the
projective resolution that suits our needs.

For example, we may choose for $P_*$ the cell complex of $EG$, the
universal $G$-space. Then one knows how to put a CW structure on
$EG\times \Delta$ so that the corresponding cell complex is just
$P_*\otimes \Delta_*$, and likewise for $EG\times Bd(\Delta)$. We see
in this fashion that
$$ \begin{array}{rcl}
H^*(\bar P_*) & \simeq  & H^*((EG\times \Delta)/G,
(EG\times Bd(\Delta))/G; \f_2)  \\
                             &  =       & H^*( E, E_0; \f_2).
\end{array}$$
Here $E$ is the (total space of) the vector bundle $(EG\times V)/G$
over $BG$, as above. Also, the upper star on $H^*(\bar P_*)$ is meant
to indicate the homology of the complex $Hom_{\f_2[G]}(\bar P_*,
\f_2)$. Therefore we have a Thom isomorphism and in particular, we have
a Thom class $T(1)$ of degree $n$ in $H^*(\bar P_*)$.

On the other hand, we may pick the minimal resolution as our $P_*$. In
this way $\bar P_*$ becomes computable.

\subsubsection{Steenrod operations.} To complete the analogy with the
 topological approach, we need to define the elements
$Sq^iT(1)$. There is indeed an algebraic definition of the Steenrod
operations, in terms of the Evens norm map: see
\cite{carlson}. Strictly speaking, it is only defined for projective
resolutions of a module, and our $\bar P_*$ is no such thing. However,
extending the operations to this case is relatively straightforward
(it is no harder than to define Steenrod operations on relative
cohomology groups given the definition on regular cohomology groups).

Thus we do have elements $Sq^iT(1)\in H^*(\bar P_*)$, and via the Thom
isomorphism they correspond to the Stiefel-Whitney classes $w_i(V)\in
H^*(G, \f_2) = H^*(P)$. Perhaps more concretely, $w_i(V)$ is
characterized as the only element in $H^*(P)$ such that the (external)
cup product $w_i(V)\cdot T(1) = Sq^iT(1)$.




\subsubsection{Pros and cons.}
On the pros side: there is no reference to any other group than $G$,
and the procedure works directly for any group, not necessarily a
$2$-group. On the cons side, {\em one needs to know $V$ rather than
  just its character}. So we need to find matrices representing the
action of each generator of $G$. Clearly, this constitutes a rather
heavy task and seems to prevent en masse calculations with lots of
groups and lots of representations.

\subsection{The finite field trick}

\subsubsection{Universal representations.}
The last idea which we present is to try and find finite groups $G_n$
for $n=1, 2, \ldots$ and for each $n$ a real representation $V_n$ of
$G_n$, such that any real representation of any group $G$ is the pull-back
of some $V_n$ under some homomorphism $f: G \to G_n$. Essentially, we
shall explain that $G_n= O_n(\f_3)$ fits the purpose (under some
hypotheses which are satisfied for us). Moreover, for a given $G$ with
a representation $V$, we can take $n= \dim V$.

Granted this, one can compute the Stiefel-Whitney classes of $V_n$ in
the cohomology of $G_n$, and pull them back using $f$. However
complicated the computations with $G_n$ may be, once they are done for
all $n\le N$ we are able to perform rapidly many computations with any
$G$ whose representations are of dimension $\le N$. Note that the
maximal dimension of an irreducible representation of $G$ grows much
more slowly than the size of $G$.

\subsubsection{Reducing mod $3$.}
We shall be concerned with real representations of real type of a
finite group $G$. We recall that a real representation $r$ of a finite
group can be of real, complex, or quaternion type. In the complex or
quaternion case, $r$ carries a complex structure, and its
Stiefel-Whitney classes can be computed from the Chern classes (see \S
\ref{sec-formal}). In order to deal with these Chern classes, one may
follow the procedure below, replacing $O_n$ by $GL_n$ and
Stiefel-Whitney by Chern throughout (details left to the reader). The
real case is the more delicate one.

Now, an irreducible real representation $V$ is of real type if and
only if its complexification is still irreducible. Alternatively, an
irreducible complex representation is the complexification of such a
real representation of real type if and only if it carries a
$G$-equivariant symmetric bilinear form. Here is a first
application. Assume from now on that $G$ is a $2$-group. Then any
irreducible complex representation is induced from a $1$-dimensional
representation of a subgroup. It follows that an irreducible real
representation $V$ of real type is induced from a real,
$1$-dimensional representation of a subgroup. As a result, we see that
$V$ may be realized with matrices with integer coefficients (indeed,
involving only $1$, $-1$ and $0$), and in such a way that $G$
preserves the symmetric bilinear form given by the identity matrix.

Therefore it makes sense to reduce all those entries mod $3$, say (any
odd prime would do). We obtain a representation $\bar V$ of $G$ over
$\f_3$, for which the standard quadratic form is $G$-invariant. Hence
we end up with a homomorphism $f: G \to O_n(\f_3)$.

The group $O_n(\f_3)$ has a canonical (defining) representation $\bar
V_n$ over $\f_3$. It is tautological that, if $V$ is as above, then
$\bar V = f^*(\bar V_n)$ (here $f^*$ means the pull-back along $f$).

\subsubsection{Going back to characteristic $0$.}
Now we use a {\em Brauer lift} of $\bar V_n$: this is a virtual
representation $V_n$ of $O_n(\f_3)$ over $\z_3$ whose mod $3$
reduction is the given $\bar V_n$. Brauer lifts always exist according
to \cite{serre}. Moreover in our case Quillen in \cite{quillenadams}
has observed that $V_n$, when viewed as a representation over $\c$,
carries an $O_n(\f_3)$-invariant quadratic form. Thus it is the
complexification of a real representation and it makes sense to speak
of its Stiefel-Whitney classes $w_i(V_n)\in H^*(O_n(\f_3), \f_2)$.

If one considers the virtual representation $f^*(V_n)$ of $G$, one
observes that its reduction mod $3$ is $f^*(\bar V_n) = \bar
V$. However, since $G$ is a $2$-group, the process of reducing mod $3$
is an isomorphism
$$ R_{\q_3}(G) \stackrel{\simeq}{\longrightarrow} R_{\f_3}(G)  $$
by \cite{serre}. It follows that $V$ and $f^*(V_n)$ are isomorphic
over $\q_3$; hence they are isomorphic over $\c$ as well; since they
are each, over the complex numbers, the complexification of a
(possibly virtual) real representation, it follows that the
corresponding real representations are isomorphic, and have thus the
same Stiefel-Whitney classes. The bottom line being that
$w_i(V)=f^*(w_i(V_n))$.

\subsubsection{Pros and cons.}
Computing the Stiefel-Whitney classes of sufficiently many $V_n$'s is
not such a tall order. First, the computation for a given $N$ yields
in fact the result for all $n\le N$ by restriction. Second, one need
not take $N$ very large: say for groups of order dividing $64$, the
dimension of a real representation of real type cannot be more than
$4$, so dealing with $O_4(\f_3)$ should be enough to treat these 340
groups (to be honest, we must recall that one must also take care of
$GL_4(k)$ for some finite field $k$ containing $\f_3$ in order to deal
with the Chern classes). Finally, we note that the mod $2$ cohomology
of $O_n(\f_3)$ is rather tractable, since it is detected on a product
of dihedral groups (which are well-understood), see \cite{fiedo}. It
{\em is} an issue in practical terms, however, that we need to find
explicit cocycles for these Stiefel-Whitney classes.

A significant disadvantage is, as above, that we need to know $V$ in
terms of matrices, rather than just its character, if we are to
compute the map $f$.